\documentclass[article,onefignum,onetabnum]{siamart190516}

\usepackage{lipsum}
\usepackage{amsfonts}
\usepackage{graphicx}
\usepackage{epstopdf}
\usepackage{algorithmic}
\ifpdf
  \DeclareGraphicsExtensions{.eps,.pdf,.png,.jpg}
\else
  \DeclareGraphicsExtensions{.eps}
\fi
\usepackage{color}
\usepackage{comment}
\usepackage{subfigure}
\usepackage{amssymb}
\usepackage{bm}

\usepackage[colorinlistoftodos]{todonotes}

\newsiamremark{remark}{Remark}
\newsiamremark{hypothesis}{Hypothesis}
\crefname{hypothesis}{Hypothesis}{Hypotheses}
\newsiamthm{claim}{Claim}

\headers{Novel splitting  integration for HMC}{F. Diele, C. Marangi, C. Tamborrino, C. Tarantino }

\title{Novel energy-preserving splitting integration for Hamiltonian Monte Carlo method
}

\author{Fasma Diele,  Carmela Marangi, \thanks{Istituto per Applicazioni del Calcolo 'M.Picone',National Research Council (CNR), via Amendola 122/D, 70126 Bari, Italy
  (\email{fasma.diele@cnr.it}, \email{carmela.marangi@cnr.it}). 
  }
 \and Cristiano Tamborrino \thanks{ Institute of Nanotechnology, National Research Council(CNR), Via Monteroni, 73100 Lecce, Italy  (\email{cristiano.tamborrino@nanotec.cnr.it})}
 \and Cristina~Tarantino
 \thanks{Institute of Atmospheric Pollution Research (IIA), National Research Council (CNR), c/o Interateneo Physics Department, Via Amendola 173, 70126 Bari, Italy   (\email{cristina.tarantino@iia.cnr.it }).} 
}

\usepackage{amsopn}

\ifpdf
\hypersetup{
  pdftitle={Novel energy-preserving splitting integration for Hamiltonian Monte Carlo method},
  pdfauthor={F. Diele, C.Marangi, C. Tamborrino, C. Tarantino}
}
\fi

\usepackage{biblatex}
\bibliography{SplittingHMC_New} 

\begin{document}

\maketitle
\begin{abstract}
Splitting schemes are numerical integrators for Hamiltonian problems that may advantageously replace the { St\"ormer-Verlet} method within Hamiltonian Monte Carlo (HMC)  methodology. 
  However, HMC performance is very sensitive to the step size parameter; in this paper we propose a new method in the one-parameter family of second-order of
splitting procedures that uses a well-fitting parameter that nullifies the expectation of the energy error for univariate and multivariate Gaussian distributions, taken as a problem-guide for more realistic situations; we also provide a new algorithm that through an adaptive choice of the $b$ parameter and the step-size ensures high sampling performance of HMC. For similar methods introduced in recent literature, by using the proposed step size selection,  the  splitting integration within HMC method  never rejects a sample  when applied to univariate and multivariate Gaussian distributions. For more general non Gaussian  target distributions the proposed approach exceeds the principal especially when the adaptive choice is used.
   The effectiveness of the proposed  is firstly tested on some benchmarks examples taken from literature. Then, we conduct experiments by considering as target distribution, the Log-Gaussian Cox process and Bayesian Logistic Regression. 

\end{abstract}



\section{Introduction}
\label{S:Intro}
In the seminal paper  \cite{duane1987hybrid},  the  two main approaches to simulate the distribution of states for a molecular system, i.e.  the Markov Chain Monte Carlo (MCMC) originated with the classical paper in \cite{metropolis1953equation} and the deterministic one, via  Hamiltonian formalism \cite{alder1959studies}, are merged in a unique method, originally
named {\it Hybrid Monte Carlo}, hereinafter referred to as {\it Hamiltonian Monte Carlo} (HMC),  taking up the suggestion made  by  R.M. Neal in \cite{neal2011mcmc}. 

\noindent  At each step of the Markov chain, HMC requires the numerical integration of a
Hamiltonian system of differential equations; typically, the second-order splitting method known as  {\it St\"ormer-Verlet} or {\it Leapfrog}   algorithm
(see, e.g.,  \cite{leimkuhler2004simulating}), is used to carry out such an integration. Whether the above algorithm may be replaced by more efficient alternatives is the question faced by many researchers (see, for example, \cite{leimkuhler2013robust}, \cite{predescu2012computationally}, \cite{joo2000instability},  \cite{takaishi2006testing}, \cite{blanes2021symmetrically} and  references therein).  In designing a new algorithm the goal is to enlarge the usable time step in order to explore larger portion of the phase space; however, working in the high-time step regime for long-time simulations induces perturbations in  computed probability, dependent on the step size. This bias  leads to a distortion in calculated energy averages which produces a high percent of rejections in HCM algorithm.

\noindent An element unifying  the recent efforts to propose alternatives to the St\"ormer-Verlet algorithm (see, for example, \cite{blanes2014numerical}, \cite{predescu2012computationally}, \cite{leimkuhler2013robust}), is the analysis of their effectiveness when applied to Gaussian distributions. 
Needless to say, as already underlined in \cite{blanes2014numerical},   it makes no practical sense to use a Markov chain algorithm to sample from a Gaussian distribution, as it makes no sense to  numerically integrate the harmonic oscillator equations. However, it is a common practice to evaluate the performance of algorithms  on simple problems as they represent benchmarks for more complex situations.

\noindent In this perspective, we propose a specific selection of the step size parameter $h_b=h(b)$ as function of the parameter $b$ defining  the one-parameter family of second-order of splitting procedures proposed in  \cite{blanes2014numerical}. When adopting  the proposed criterion  for sampling from Gaussian distributions,  all the methods in the splitting  family are featured by a zero expectation value for the 
 random variable representing  the  energy error. The novel approach stems from some energy-preserving splitting methods for Hamiltonian dynamics proposed in  \cite{pace2015splitting},  here adapted in the context of HMC. Specifically, instead of fixing the step size $h$ and choosing the  parameter $b$  which minimizes the expectation of the energy error as in  \cite{blanes2014numerical}, we fix the parameter $b$ and we identify the step size $h_b$ which exactly  nullifies the energy error and, consequently, its expected value. 
 
 \noindent As the distortion in calculated energy averages produces a high percent of rejections in HCM algorithm,  preserving as much as possible the energy is of  outmost importance within the HMC procedure in terms of saving of computational time, particularly in the case of high-dimensional problems \cite{calvo2019hmc}.
For the above reasons, in this paper we explore whether the adopted step size selection  which nullifies the energy error in case of the both univariate and multivariate test problems, can also reduce the number of rejection steps even when used within HMC processes  for sampling from generic distributions.
Moreover, we propose a novel implementation of the HMC algorithm based on  an adaptive choice of the parameter $b$, defining  the one-parameter family of second-order of splitting procedures, based on its  reduction whenever a sample is rejected. 
In particular, we test our technique on the Log-Gaussian Cox model, a point  process for presence-only species distribution representing a   statistical tool supporting the modelling of the spread of invasive species \cite{renner2015, baker2018optimal, baker2019optimal, lacitignola2015dynamical}. 
 
\noindent The presentation of the step size selection for the family of  splitting integrators here considered is  preceded by the analysis of the linear map generated by the application of a general volume-preserving   and momentum  flip-reversible integrator within the HMC method. The obtained results  generalize the ones given in \cite{blanes2014numerical} in that the standard deviations may assume arbitrary values and, consequently,  more general expressions for both the energy error and of its expected value are provided; on the other hand,  it departs from  \cite{blanes2014numerical}, this representing an adding element of novelty, as the quantity responsible for the generation of the error in approximating the Hamiltonian, is here exactly identified. 
  In doing so, when analyzing the special family of second-order splitting integrators, it turns out to be a trivial task to identify  the  parameter $h$  that makes the resulting methods   exactly  energy preserving. 
 



\bigskip 
\noindent The paper is organized as follows: in Section \ref{sec:2}  the general framework on sampling from a target distribution throughout the HMC algorithm is recalled and some theoretical and practical implementation details are briefly provided. In Section \ref{sec:3} we present the analysis of the energy- preserving linear maps generated by the application of a general  volume-preserving   and momentum  flip-reversible integrator within HMC method on  both univariate and multivariate Gaussian distributions, taken as test problems. In Section \ref{sec:4}, the  maps built on the splitting of the Hamiltonian vector field are introduced; then the classical  St\"ormer-Verlet method (Section \ref{sec:4.1}) and the one-parameter family of second-order splitting integrators (Section \ref{sec:4.2}) are presented. For the last class, the main result is described in Theorem \ref{Th2} where a suitable selection of the step size provides a stable, energy preserving approximation of the univariate Gaussian test problem.   This result is then generalized  for multivariate Gaussian distributions in Section \ref{sec:4.3} (Theorem \ref{th:multi}). For sampling from generic distributions, in Section \ref{sec:5}  we propose, within  the HMC algorithm, the novel implementation of a one-parameter family of  splitting procedures  which advances with the same step size $h=h_b$ which nullifies the energy error in the case of Gaussian distributions, adopting an adaptive reduction of the parameter $b$ as presented in Algorithm \ref{alg:buildtreenovel}. 
Numerical experiments are given in Section \ref{sec:6}. As a verification of the theoretical results, we firstly apply the proposed energy-preserving procedure for both bivariate and multivariate Gaussian distributions.
Then, we show the performance  for   more general distributions within the class of perturbed Gaussian models, i.e. Log-Gaussian Cox processes, representing the distribution of the {\it Ailanthus altissima} tree, an invasive alien species spreading in a protected area in the South of Italy \cite{baker2018optimal}, \cite{baker2019optimal}, \cite{marangi2020mathematical} and for the Bayesian Logistic Regression model.  Conclusive remarks and  possible future developments are drawn in Section \ref{sec:conclusions}.

\section{The Hamiltonian Monte Carlo algorithm}\label{sec:2}
The  description of the HMC algorithm given below  follows the steps described in \cite{neal2011mcmc}.  Given a data set  $X$, suppose that we wish to sample, using Hamiltonian dynamics,  the variable  $\mathbf{q}\, \in \mathbf{R}^d$ from a probability distribution of interest  $\mathcal{P}(\mathbf{q})$ with prior density $\pi(\mathbf{q})$ and likelihood function $L(\mathbf{q}|X)$ i.e. $\mathcal{P}(\mathbf{q})\, = \,\displaystyle  \pi(\mathbf{q}) \, L(\mathbf{q}|X) $. The first step is to associate, via the canonical distribution, a potential energy function defined as follows 
$$
U(\mathbf{q})\, =\, -\log\left[\mathcal{P}(\mathbf{q})\right]\, -\, \log(Z), \quad Z>0,
$$
so that 
$
\mathcal{P}(\mathbf{q})\, \propto \, \, \,  \exp\,( -U(\mathbf{q})).
$
Then, we introduce  auxiliary  momentum variables $\mathbf{p}\in \mathbf{R}^d$, independent of $\mathbf{q}$, specifying the distribution via the kinetic energy function $K(\mathbf{p})$. The current practice with HMC is to use a quadratic kinetic energy
$
K(\mathbf{p})\, =\, \frac{1}{2} \, \mathbf{p}^T \, D_{\beta}^{-1}\mathbf{p}
$
where, without loose of generality we suppose that the components of $\mathbf{p}$ are specified to be independent so that 
 $D_{\beta}$ is a diagonal matrix 
 with entries $\beta^2_i$, each representing the variance of the $ith$  component $p_i$ of the vector $\mathbf{p}$. 
The canonical distribution 
$
\mathcal{P}(\mathbf{p}) \, =\, \exp(-K(\mathbf{p})) \,
$
results to be the  zero-mean multivariate Gaussian distribution. We denote with  $H(\mathbf{q},\, \mathbf{p})\,=\,  U(\mathbf{q})\, +\, K(\mathbf{p})$ the energy function for the joint state of position $\mathbf{q}$ and momentum $\mathbf{p}$, which defines a joint canonical distribution satisfying
$$
\mathcal{P}(\mathbf{q}, \mathbf{p })\, =\, \frac{1}{Z} \exp(-H(\mathbf{q}, \, \mathbf{p}))\, =\, \frac{1}{Z} \exp(\,-U(\mathbf{q}))\,   \exp(-K( \mathbf{p}))\, = \, \mathcal{P}(\mathbf{q})\,\mathcal{P}(\mathbf{p}).
$$
We see that the joint (canonical) distribution for $\mathbf{q}$ and $\mathbf{p}$  factorizes. This means that the two variables are independent, and the canonical distribution $\mathcal{P}(\mathbf{q})$ is independent of $\mathcal{P}(\mathbf{p})$. Therefore, we can use the Hamiltonian dynamics to sample from the joint canonical distribution $\mathcal{P}(\mathbf{q}, \mathbf{p })$  and simply ignore the momentum contributions. The introduction of  the auxiliary variable $\mathbf{p}$  allows the Hamiltonian dynamics to perform \cite{neal2011mcmc}.

\noindent Starting from the generation of  an initial position state $\mathbf{q}^{(i)}  \, \propto \pi(\mathbf{q})$, for $i=0,\dots L$ each iteration of the HMC algorithm has two steps. The first step  chooses the initial momentum by randomly drawing  values $\mathbf{p}^{(i)}$ from its zero-mean multivariate Gaussian distribution $\mathcal{N}(0,D_{\beta})$.  The second step, 
starting at $t=0$ with initial states $\mathbf{Q}(0) \, =\,  \mathbf{q}^{(i)}$ and $\mathbf{P}(0) \, =\,  \mathbf{q}^{(i)}$  solves the  Hamiltonian dynamics 
\begin{equation}\label{eq:HS}
  \dfrac{d\mathbf{Q}}{dt} =\,\nabla_\mathbf{P} \, K(\mathbf{P})\, = \,  D_{\beta}^{-1}\, \mathbf{P}, \qquad \dfrac{d\mathbf{P}}{dt}\, =\, - \nabla_\mathbf{Q} \, U(\mathbf{Q}), \qquad t\in (0, \, T^*].   
\end{equation}
with Hamiltonian function
\begin{equation}\label{eq:Hami_general}
    H(\mathbf{Q},\mathbf{P}): = \frac{1}{2} \, \mathbf{P}^T \, D_{\beta}^{-1} \mathbf{P} \, +\,  U(\mathbf{Q})
\end{equation}
Then,  the state of the position  at the end of the simulation $\mathbf{Q}(T^*)$  is used as the next state of the Markov chain by setting   $\mathbf{q}^{(i+1)}\, = \mathbf{Q}(T^*)$. Combining these steps, the sampling of the random momentum, followed by the Hamiltonian dynamics, defines the {theoretical} HMC \cref{alg:buildtree} for drawing $L$ samples from a target distribution.

\begin{algorithm}
\caption{HMC algorithm (theoretical)}
\label{alg:buildtree}
\begin{algorithmic}
\STATE{Draw $\mathbf{q}^{(1)} \sim \pi(\mathbf{q})$,  $\mathbf{q}^{(1)}\in \mathbb{R}^d$,  $L\geq 1$, set  $i=0$}
\WHILE{$i < L$}
\STATE{i=i+1}
\STATE{Draw  $\mathbf{p}^{(i)} \sim \mathcal{N}(0,D_{\beta})$, }
\STATE{Set $(\mathbf{Q}(0), \,\mathbf{P}(0) ) =  (\mathbf{q}^{(i)},\, \mathbf{p}^{(i)})$, set j\,=\,0}
\WHILE{$j < 1$}
\STATE {Randomly  choose $T^*>0$}
\STATE{Solve $\dfrac{d\mathbf{Q}}{dt} =D_{\beta}^{-1}\, \mathbf{P}, \quad \dfrac{d\mathbf{P}}{dt}\, = - \nabla_\mathbf{Q} \, U(\mathbf{Q}), \quad t\in (0, \, T^*]$}
\STATE{if $\left(\mathbf{Q}(T^*),\, \mathbf{P}(T^*)\right)\, \neq \, (\mathbf{Q}(0), \,\mathbf{P}(0) )$, j\,=\,1 }
\ENDWHILE
\STATE{Update:  $\mathbf{q}^{(i+1)}\, = \mathbf{Q}(T^*)$ }
\ENDWHILE
\RETURN  Markov chain $\mathbf{q}^{(1)}, \, \mathbf{q}^{(2)}, \dots, \, \mathbf{q}^{(L)}$
\end{algorithmic}
\end{algorithm}
\noindent  As anticipated in the Introduction, an important observation is that, in the framework of Hamiltonian dynamics for the Markov Chain Monte Carlo algorithm, the fictitious final time  $T^*>0$ behaves as a parameter to be selected. A criterion adopted to select this value should preserve the {\it ergodicity} of the  HMC algorithm.
In a HCM iteration, any value can be sampled for the momentum variables, which  can typically then affect the position variables in arbitrary ways; however, ergodicity can fail if the chosen $T^*$ produces  an exact periodicity for some function of the state. For example, with ${q}^{(i)} \sim \mathcal{N}(0,1)$ and ${p}^{(i)} \sim \mathcal{N}(0,1)$, the Hamiltonian dynamics for $Q$ and $P$ define the equations of  harmonic oscillator 
\begin{equation}\label{arm_osctest}
 \dfrac{dQ}{dt}=P, \qquad
 \dfrac{dP}{dt}\,=-Q,
\end{equation}
whose solutions are periodic with period $2\, \pi$. Choosing $T^*\, =2\, \pi$ the trajectory returns to the same position coordinate and the HCM will be not ergodic.  
This potential problem of non-ergodicity can be solved by randomly choosing $T^*$ and doing this routinely, as in \cref{alg:buildtree}.


\subsection{Practical  implementation of the HMC algorithm}
\noindent Starting from $\mathbf{Q}_0=\mathbf{Q}(0)$, $\mathbf{P}_0=\mathbf{P}(0)$,  a practical implementation of  \cref{alg:buildtree} needs to numerically integrate the Hamiltonian system (\ref{eq:HS}) by means of a map $(\mathbf{Q}_{n+1},\, \mathbf{P}_{n+1})\, =\, \Psi_h(\mathbf{Q}_{n},\, \mathbf{P}_{n}) $, for $n=0,\dots N$,  where $N$ and  the step size $h$ satisfy $N\,h\, = T^*$. 
In order to safely  replace the theoretical solution with an approximated one, the chosen map $\Psi_h$ should result a transformation in phase space  which inherits, from the theoretical flow, two main characteristic: to be volume-preserving i.e. 
$
\det(\Psi_h'(\mathbf{Q}_n,\mathbf{P}_n))\,=\, 1, 
$ where
$\Psi'$ denotes the Jacobian matrix of $\Psi$,
and {\it momentum flip} - reversible \cite{hairer2003geometric}:
$$
\Psi_h(\mathbf{Q}_{n},\, \mathbf{P}_{n}) = (\mathbf{Q}_{n+1},\, \mathbf{P}_{n+1}) \iff 
\Psi_h(\mathbf{Q}_{n+1},\, -\mathbf{P}_{n+1})\, = (\mathbf{Q}_{n}, -\mathbf{P}_{n}) 
$$
for $n=0,\dots N$.
This guarantees the construction of a  Markov chain  which is reversible with respect to the target probability distribution $\pi(\mathbf{q})$ \cite{blanes2014numerical}.

\noindent Position and momentum variables at the end of the simulation are used as   proposed  variables $\mathbf{q}^*\, =\, \mathbf{Q}(T^*)$  and  $\mathbf{p}^*\,=\, \mathbf{P}(T^*)$ and are  accepted using an update rule analogous to the Metropolis acceptance criterion. Specifically, if the probability of the joint distribution at $T^*$ i.e. $exp(-H(\mathbf{q}^*, \, \mathbf{p}^*))$ is greater then the initial $exp(-H(\mathbf{q}^{(i)}, \, \mathbf{p}^{(i)}))$, then the proposed state  is accepted  and  $\mathbf{q}^{(i+1)}\, = \mathbf{q}^*$, otherwise it is rejected and  the next state of the Markov chain is set as  $\mathbf{q}^{(i+1)}\, = \mathbf{q}^{(i)}$. Combining these steps, sampling random momentum, followed by Hamiltonian dynamics and Metropolis acceptance criterion, defines the HMC  \cref{alg:buildtreepractical} for drawing $L$ samples from a target distribution.

\begin{algorithm}
\caption{HMC algorithm (practical)}
\label{alg:buildtreepractical}
\begin{algorithmic}
\STATE{Draw $\mathbf{q}^{(1)} \sim \pi(\mathbf{q})$,  $\mathbf{q}^{(1)}\in \mathbb{R}^d$, $L\geq 1$, set  $i=0$}
\WHILE{$i < L$}
\STATE{i=i+1}
\STATE{Draw  $\mathbf{p}^{(i)} \sim \mathcal{N}(0,D_{\beta})$}
\STATE{Set $(\mathbf{Q}_0, \,\mathbf{P}_0 ) =  (\mathbf{q}^{(i)},\, \mathbf{p}^{(i)})$}, set $j\,=\,0$
\WHILE{$j < 1$}
\STATE {Randomly  choose $T^*>0$}.
\STATE{Set $N\geq 1$ or  $h>0$ such that $T^*\, =\, N\, h$} 
\STATE{Evaluate $(\mathbf{Q}_{n+1}\, \mathbf{P}_{n+1})\, =\, \Psi_h(\mathbf{Q}_{n},\, \mathbf{P}_{n}) $, \,for $n=0,\dots N-1$}
\STATE{if $\left(\mathbf{Q}_N,\, \mathbf{P}_N\right)\, \neq \, (\mathbf{Q}_0, \,\mathbf{P}_0 )$, j\,=\,1 }
\ENDWHILE
\STATE{Set $\left(\mathbf{q}^*,\, \mathbf{p}^*\right)\,=\, \left(\mathbf{Q}_N,\, \mathbf{P}_N\right)$ }
\STATE{Calculate 
$\alpha = \text{min}\left(1,\exp\left(H(\mathbf{q}^{(i)},\mathbf{p}^{(i)}) -H(\mathbf{q}^*,\mathbf{p}^*)\right)\right)$}
\STATE{Draw  $u \sim \mathcal{U}(0,1)$}
\STATE{Update: if $\alpha >\,u$ then $\mathbf{q}^{(i+1)}\, = \mathbf{q}^*$; otherwise $\mathbf{q}^{(i+1)}\, = \mathbf{q}^{(i)}$}
\ENDWHILE
\RETURN  Markov chain $\mathbf{q}^{(1)}, \, \mathbf{q}^{(2)}, \dots, \, \mathbf{q}^{(L)}$
\end{algorithmic}
\end{algorithm}

\noindent Notice that a map which approximates the solution of the Hamiltonian flow (\ref{eq:HS}), in such a way that $H(\mathbf{Q}_N,\mathbf{P}_N)\,-\, H(\mathbf{Q}_0,\mathbf{P}_0) \, \leq\, 0$   produces all accepted proposals.  However, in \cite{blanes2014numerical} it has been shown that, roughly speaking,  for a
momentum-flip reversible volume-preserving transformation the phase space is always  divided into two regions
of the same volume, one corresponding to negative energy errors  
and the other, corresponding to flip the momentum 
with positive energy errors, so that, unless the map is energy-preserving 
it  may potentially lead to rejections.

\section{Energy-preserving linear maps for Gaussian  distributions}\label{sec:3}
\noindent The chosen  step size $h>0$ is  crucial  in the implementation of \cref{alg:buildtreepractical}.  Too small a step size will  waste computation time as it will require a large $N$ in order to reach the final step  $T^*\, =N\, h$.  Too large a step size will increase bounded oscillations in the value of the Hamiltonian, which would be constant if the trajectory were simulated by an energy-preserving map. Moreover, when values for $h$ are chosen above the critical stability threshold, which is characteristic of each approximating map $\Psi_h$, then the  Hamiltonian grows without bound, resulting to  an extremely low acceptance rate for states proposed by simulated trajectories. 
Hence the selection of the step size $h$ should obey to stability constraints.  The issue of stability is traditionally faced by means of a test problem; for HCM flows,  it is represented by the problem defined by a Gaussian  zero-mean distribution for both $q$ and $p$. Firstly, we account for the one-dimensional  problem and then we extend the analysis to the multi-dimensional case.  
\subsection{Univariate case}\label{sec:3.1}
We generalize the approach in both  \cite{blanes2014numerical} and \cite{neal2011mcmc} by considering  generic standard deviations, ${\alpha}$ for  $q$  and  ${\beta}$ for $p$, with zero correlation. 
The Hamiltonian dynamics for $Q$ and $P$ define the equations 
\begin{equation}\label{eq:arm_general}
 \dfrac{dQ}{dt}\,=\, \dfrac{P}{\beta^2}, \qquad
 \dfrac{dP}{dt}\,=-\dfrac{Q}{\alpha^2}.
\end{equation}
Setting $\mathbf{Y}\,=[Q,\,P]^T$,  the Hamiltonian can be expressed as  $H(\mathbf{Y}) =\,\dfrac{1}{2} \, \mathbf{Y}^T \,
\mathcal{D}_2^{-1}
\mathbf{Y}\, =\, \dfrac{1}{2} \left(\dfrac{Q^2}{\alpha^2}\, +\, \dfrac{P^2}{\beta^2}\right)$  where 
$
\mathcal{D}_2\,: = \left[\begin{array}{lc}
 \alpha^2   & 0 \\
    0& \beta^2
\end{array}\right].
$ Starting from $Q_0,\, P_0$, the theoretical solution at $t_n= \, n\,h$ is represented as a  linear map $\mathbf{Y}(t_n)\, =\, {\mathcal{F}}^{(n\,h_\sigma,\, \sigma)}\, \mathbf{Y}_{0}$, where
$$
{\mathcal{F}}^{(n\, h_\sigma,\, \sigma)}\,:=\, \left[\begin{array}{lc}
   \cos\left(n \, h_\sigma\right)   & \sigma^{-1}  \,\sin\left({ n\,h_\sigma}\right)\,   \\\\
      -\sigma  \,\sin\left({ n\, h_\sigma}\right)\,&  \cos{\left( n\, h_\sigma\right)}   \end{array}\right],  \quad \sigma\, :=\, \dfrac{\beta}{\alpha}, \quad h_\sigma:=\dfrac{h}{\alpha\, \beta}.
$$

\noindent Notice that  Hamiltonian can be expressed as  $H(\mathbf{Y}) =\,\dfrac{1}{2\, \alpha\, \beta} \,   \left(\sigma \,Q^2 \,+\, \dfrac{P^2}{\sigma}\, \right)$.

\bigskip
\noindent We mentioned that the numerical map used to replace the theoretical solution with an approximation should  be volume-preserving (here equivalent to symplectic) and  momentum be flip-reversible. Both characteristics direct  our attention to the class  of integrators that,  when applied to the test problem (\ref{eq:arm_general}), can be expressed  as $$\mathbf{Y}_{n+1}\, =\, \mathcal{M}_2^{( h, \sigma)}\, \mathbf{Y}_{n}$$ where 
$\mathcal{M}_2^{( h, \sigma)}(1,1)\, =\, \mathcal{M}_2^{( h, \sigma)}(2,2)$
and $\det(\mathcal{M}_2^{( h, \sigma)})\,=\,1$.

\noindent Setting    $\mathrm{p}_{ h}=\mathcal{M}_2^{( h, \sigma)}(1,1)\, =\, \mathcal{M}_2^{( h, \sigma)}(2,2)$, \,   $\mathrm{q}_{ h}= \dfrac{\sigma}{\sigma^2+1}\,(\mathcal{M}_2^{(h, \sigma)}(1,2)\,-\, \mathcal{M}_2^{( h, \sigma)}(2,1))$ and 
$\mathrm{e}_{ h}\, = \, \dfrac{1}{\sigma^2+1}\,(\sigma^2 \mathcal{M}_2^{( h, \sigma)}(1,2)\,+\, \mathcal{M}_2^{( h, \sigma)}(2,1))
$, the matrix $\mathcal{M}_2^{( h, \sigma)}$ can be written as 
 \begin{equation} \label{eq:tildeM}
     \mathcal{M}_2^{{( h},\sigma)}\,=\, \left[\begin{array}{lc}
   \mathrm{p}_{ h}   &\mathrm{e}_{ h}\, +\, \sigma^{-1}  \,\, \mathrm{q}_{ h}  \\\\
      \mathrm{e}_{ h}\, -\, \sigma \,\, \mathrm{q}_{ h}&  \mathrm{p}_{ h}   \end{array}\right],
\end{equation}  
and, from $\det(\mathcal{M}_2^{( h, \sigma)})\, =\,1$,  the following relation holds 
\begin{equation}\label{det1}
\mathrm{p}_{ h}^2\,-\, (\mathrm{e}_{ h}\, +\, \sigma^{-1}  \, \mathrm{q}_{ h})\, ( \mathrm{e}_{ h}\, -\, \sigma \,\, \mathrm{q}_{ h})\, =\,1.    
\end{equation}

 \noindent The stability of the trajectories depends on eigenvalues of $\mathcal{M}_2^{( h, \sigma)}$
which solve the polynomial 
$$
\lambda^2\,-\, 2\, \mathrm{p}_{h}\,\lambda  +\,1\, =0.
$$
When $\mathrm{p}_{ h}^2 \,-\, 1\, \geq 0$ then the eigenvalues are real with at least one having absolute value greater than one, hence the trajectories are unstable. When $\mathrm{p}_{ h}^2 \,-\, 1\, <\, 0$ the eigenvalues are complex with modulus equal to one, hence the trajectories are stable.
\bigskip

\noindent The key consideration for what follows is that integrators for which it results  $\mathrm{e}_{h}\, =\,0$ are energy preserving. 
Indeed, the  error in energy at each step is given by 
$$
    \begin{array}{ccl}
       \Delta_2^{(n, h)}\,&:=  & H(\mathbf{Y}_{n+1})\, - H(\mathbf{Y}_n)\,=\,\dfrac{1}{2} \, \mathbf{Y}_{n+1}^T \, \mathcal{D}_2^{-1}\, \mathbf{Y}_{n+1}\,-\, \dfrac{1}{2} \, \mathbf{Y}_n^T \, \mathcal{D}_2^{-1}\,\, \mathbf{Y}_n\, \\\\
         &= &
         \dfrac{1}{2} \,\mathbf{Y}_n^T \, {\mathcal{M}_2^{( h, \sigma)}}^{T}\,\mathcal{D}_2^{-1}\,\,\,{\mathcal{M}_2^{( h, \sigma)}}\, \mathbf{Y}_{n}\,-\,\dfrac{1}{2} \, \mathbf{Y}_n^T \, \mathcal{D}_2^{-1}\,\, \mathbf{Y}_n\\\\\
        &=&  \dfrac{1}{2} \, \mathbf{Y}_n^T \,\left({\mathcal{M}_2^{( h, \sigma)}}^{T}\,\mathcal{D}_2^{-1}\,\,{\mathcal{M}_2^{( h, \sigma)}}\, -\, \mathcal{D}_2^{-1}\, \right)  \mathbf{Y}_n\\\\
        &=&
        \dfrac{1}{2} \, \mathbf{Y}_n^T \,\left({\mathcal{K}_{2}^{(h)}}^{T}\,\mathcal{K}_{2}^{(h)}\, -\, \mathcal{D}_2^{-1}\,\right)  \mathbf{Y}_n
    \end{array}
$$
where $\mathcal{K}_{2}^{(h)}\, =\, \mathcal{D}_2^{-1/2}\,{\mathcal{M}_2^{( h, \sigma)}} \, =\,\left(\begin{array}{cc}  \dfrac{\mathrm{p}_{ h}}{\alpha}  & \dfrac{\mathrm{e}_{ h}}{\alpha}  \,+\, \dfrac{\mathrm{q}_{ h}}{\beta}  \\\\ \dfrac{\mathrm{e}_{ h}}{\beta}  \,-\, \dfrac{\mathrm{q}_{ h}}{\alpha}&  \dfrac{\mathrm{p}_{ h}}{\beta}    \end{array}\right).$
Let us evaluate 
\begin{equation}\label{eq:matcalE}
{\mathcal{E}_{ 2}^{(h)}}\, =\,{\mathcal{K}_{2}^{(h)}}^{T}\,\mathcal{K}_{2}^{(h)}- \mathcal{D}_2^{-1}=\left(\begin{array}{cc} \dfrac{\mathrm{p}_{ h}^2\,-\,1}{\alpha^2}+ \left( \dfrac{\mathrm{e}_{ h} }{\beta}- \dfrac{\mathrm{q}_{ h} }{\alpha}\right)^2 & \left(\dfrac{1}{\alpha^2}+ \dfrac{1}{\beta^2}\right)\, \mathrm{e}_{ h}\,\mathrm{p}_{ h} \\\\ \left(\dfrac{1}{\alpha^2}+ \dfrac{1}{\beta^2}\right)\, \mathrm{e}_{ h}\,\mathrm{p}_{ h}  & \dfrac{\mathrm{p}_{ h}^2-1}{\beta^2} + \left( \dfrac{\mathrm{e}_{ h} }{\alpha}+ \dfrac{\mathrm{q}_{ h} }{\beta}\right)^2 \end{array}\right)
    \end{equation}
so that  
$\Delta_2^{(n,h)}\, = \dfrac{1}{2} \, \mathbf{Y}_n^T \,\mathcal{E}_{2}^{ (h)}\, \mathbf{Y}_n$, for $n=0,\dots,\,N$
and 
\begin{equation}\label{eq:expect}
    \Delta_2^{(N)}\, := H(\mathbf{Y}_N)\, - H(\mathbf{Y}_0)\, =\, \displaystyle \sum_{n=0}^N \Delta_2^{(n,h)} \,= \dfrac{1}{2} \, \displaystyle \sum_{n=0}^N \mathbf{Y}_n^T \,\mathcal{E}_{2}^{(h)}\, \mathbf{Y}_n. 
    \end{equation}
\begin{theorem}\label{thm:univariatecase}
Consider the Hamiltonian test problem (\ref{eq:arm_general}) and a   symplectic and  momentum flip - reversible integrator which  can be expressed  as 
$\mathbf{Y}_{n+1}\, =\, \mathcal{M}_2^{( h, \sigma)}\, \mathbf{Y}_{n}$ with $\mathcal{M}_2^{( h, \sigma)}$ defined in (\ref{eq:tildeM}),
when applied to  (\ref{eq:arm_general}). If it results that  $\mathrm{e}_{ h}\, =\, 0$, then the integrator preserves the Hamiltonian.
\end{theorem}
\begin{proof}
It is enough to observe that, whenever  $\mathrm{e}_{ h}\, =\, 0$, the matrix ${\mathcal{E}_{ 2}^{(h)}}$ in (\ref{eq:matcalE}) has null entries on the right-left diagonal. From relation (\ref{det1}), it follows that on the principal diagonal  ${\mathcal{E}_{ 2}^{(h)}}(1,1)\, =\, \dfrac{\mathrm{p}_{ h}^2\,+\, \mathrm{q}_{ h}^2-\,1}{\alpha^2}\,=\, {\mathcal{E}_{ 2}^{(h)}}(2,2)\, =\, \dfrac{\mathrm{p}_{ h}^2\,+\, \mathrm{q}_{ h}^2-\,1}{\beta^2}\, =\,0$ which completes the proof. 
\end{proof}

\begin{theorem}\label{eq:expectation1}
Assume that $Q_0$, $P_0$ are two random variables with Gaussian zero-mean distribution, standard deviations $\alpha$ and $\beta$ respectively and zero correlation. Suppose that the Hamiltonian dynamics (\ref{eq:arm_general}) is approximated by means of a linear map $\mathbf{Y}_{n+1}\, =\, \mathcal{M}_2^{( h, \sigma)}\, \mathbf{Y}_{n}$  with $\mathcal{M}_2^{( h, \sigma)}$ given in (\ref{eq:tildeM}). Then, the expectation of the random variable $\Delta_2^{(N)}$ in (\ref{eq:expect}) 
is given by 
$$
\mathbb{E}(\Delta_2^{(N)})\,=\,\dfrac{N}{2}\, \left(\sigma\,+\, \dfrac{1}{\sigma}\right)^2\, \mathrm{e}^2_{ h}
$$
and, consequently,
$\mathbb{E}(\Delta_2^{(N)})\,=\,0$  iff  $\mathrm{e}_{ h}\,=\,0$.
\end{theorem}

\begin{proof}
From  $\Delta_2^{(n, h)}\, =\, \dfrac{1}{2} \, \mathbf{Y}_n^T \,\mathcal{E}_{2}^{(h)}\, \mathbf{Y}_n$, we can evaluate 
$$
\begin{array}{rcl}
  2\, \Delta_2^{(n, h)}\,    & = & \left[\dfrac{\mathrm{p}_{ h}^2\,-\,1}{\alpha^2}\, \,+\, \left( \dfrac{\mathrm{e}_{ h} }{\beta}\,-\, \dfrac{\mathrm{q}_{ h} }{\alpha}\right)^2\right] Q_0^2\,+\,2\,  \mathrm{e}_{ h}\,\mathrm{p}_{ h}
\left(\dfrac{1}{\alpha^2}\,+\, \dfrac{1}{\beta^2}\right)\, \, Q_0\, P_0 \\\\
 &+&\left[\dfrac{\mathrm{p}_{ h}^2\,-\,1}{\beta^2}\, \,+\, \left( \dfrac{\mathrm{e}_{ h} }{\alpha}\,+\, \dfrac{\mathrm{q}_{ h} }{\beta}\right)^2\right] \,  P^2_0 
\end{array}
$$ 
for $n=0,\dots,\,N$. From  $\mathbb{E}(Q^2_0)\, =\, \alpha^2$, $\mathbb{E}(P^2_0)\, =\, \beta^2$, $\mathbb{E}(Q_0\, P_0)\, =\, 0$, it results that 
$$
 2\,  \mathbb{E}(\Delta_2^{(n,h)})\,  = \,  2\left(\mathrm{p}_{ h}^2\,-\,1\right)\,+\, \left( \sigma^{-1} \, \mathrm{e}_{ h} \,-\,\mathrm{q}_{ h} \right)^2 \, + \left( \sigma \, \mathrm{e}_{ h}  \,+\,\mathrm{q}_{ h} \right)^2 \, = \,\left( \sigma\, \mathrm{e}_{ h}\,+\, \sigma^{-1}\, \mathrm{e}_{ h} \right)^2  
$$ 
and the statement trivially  follows.
\end{proof}

\bigskip
\noindent The following result generalizes Proposition $4.3$ in \cite{blanes2014numerical} for Gaussian zero-mean distributions with generic standard deviations $\alpha$ and $\beta$:
\begin{theorem}\label{thm:expectation2}
Under the same hypothesis of Theorem \ref{thm:univariatecase}, assuming  $|\mathcal{M}_2^{(h, \sigma)}|<1$, the expectation of the random variable $\Delta_2^{(N)}$ in (\ref{eq:expect}) can be expressed as 
$$
\mathbb{E}(\Delta_2^{(N)})\,=\,N \sin^2{( h_{\chi_h})} \, \rho(h), \qquad \rho(h)\,=\, \dfrac{1}{2}\, \left({\tilde \chi}_h\, -\, \dfrac{1}{{\tilde \chi}_h} \right)^2,\, \qquad {\tilde \chi_h}\, := \sigma\, \chi_h^{-1} .
$$
\end{theorem}
\begin{proof}
Under the assumption $|\mathcal{M}_2^{( h, \sigma)}|\, <1$ then $|\mathrm{p}_{ h}|\,<\, 1 $ and we can define  $h_{\chi_h}\,=\, \arccos{\mathrm{p}_{ h}}, \, h_{\chi_h} \in [0,\, \pi]$ and $\sin{h_{\chi_h}}\, =\, \sqrt{1\,-\,\mathrm{p}^2_{ h}}$. From $\sin^2{( h_{\chi_h})}\,=\, 1\, -\, \mathrm{p}^2_{ h}$, and exploiting  the relations $$ \dfrac{1}{\chi_h} \,=\, \dfrac{\mathrm{e}_{ h} \,+\, \sigma^{-1} \, \mathrm{q}_{ h}}{\sqrt{1\, -\, \mathrm{p}^2_{ h}}}, \qquad \chi_h  =\, \dfrac{\sigma \, \mathrm{q}_{ h}-\mathrm{e}_{ h}}{ \sqrt{1\, -\, \mathrm{p}^2_{ h}}}, $$ which are both satisfied from (\ref{det1}),  we can prove that  
$$2\, \sin^2{( h_{\chi_h})}\, \rho(h)\,=\sin^2{( h_{\chi_h})} \left( \dfrac{\sigma}{ { \chi}_h} - \dfrac{{ \chi}_h}{ \sigma} \right)^2\,  =\, \left(\sigma\,+\, \dfrac{1}{\sigma}\right)^2\, \mathrm{e}^2_{ h}. $$ 
From Theorem \ref{eq:expectation1} the result follows.     \end{proof}

\subsection{Multivariate case}\label{sec:3.2}
The motion of $d$ oscillators
\begin{equation}\label{eq:arm_generalmulti}
 \dfrac{dQ_j}{dt}\,=\, \dfrac{P_j}{\beta_j^2}, \qquad
 \dfrac{dP_j}{dt}\,=-\dfrac{Q_j}{\alpha_j^2}, \qquad  \text{for} \quad j=1,\dots d,
\end{equation}
can be represented as an Hamiltonian system  
\begin{equation}\label{arm_osc_multgen}
\dfrac{d\mathbf{Q}}{dt}\,=\, D^{-1}_{\beta}  \mathbf{P},\qquad 
 \dfrac{d\mathbf{P}}{dt}\,=\, -D^{-1}_{\alpha}  \mathbf{Q} 
\end{equation}
 where $D_{\alpha}$ and $D_{\beta}$ are $d\, \times d$ diagonal matrices with entries $\alpha_j^2$ and $\beta_j^2$, respectively, for $j=1,\dots,\, d$ 
 and Hamiltonian function  $\dfrac{1}{2} \, \mathbf{Q}^T \, {D}_{\alpha}^{-1}\, \mathbf{Q} + \dfrac{1}{2} \, \mathbf{P}^T \,  {D}_{\beta}^{-1}\, \mathbf{P}.$ 
 Setting $\mathbf{Y}\,=[\mathbf{Q},\,\mathbf{P}]^T$, $
\mathcal{D}_{2d}\, := \left[\begin{array}{lc}
    D_{\alpha}  & \textbf{0}_d\\
  \textbf{0}_d& D_\beta
  \end{array}\right]$ the Hamiltonian can be written as  
  \begin{equation}\label{eq:Hmulti}
      H(\mathbf{Y}) =\,\dfrac{1}{2} \, \mathbf{Y}^T \,  \mathcal{D}_{2d}^{-1}\, \mathbf{Y}.  \end{equation}

\noindent   Define $\Sigma\, = \, D^{1/2}_{\beta}\, D^{-1/2}_{\alpha}$;
then a    symplectic and  momentum flip - reversible integrator for the  $d$-dimensional system (\ref{arm_osc_multgen}) can be expressed  as $\mathbf{Y}_{n+1}\, =\, \mathcal{M}_{2d}^{( h, \Sigma)}\, \mathbf{Y}_{n}$ where
\begin{equation} \label{eq:tildeMmulti}
     \mathcal{M}_{2d}^{{( h},\Sigma)}\,=\, \left[\begin{array}{lc}
   \mathrm{P}_{ h}   &\mathrm{E}_{ h}\, +\, \Sigma^{-1}  \,\, \mathrm{Q}_{ h}  \\\\
      \mathrm{E}_{ h}\, -\, \Sigma \,\, \mathrm{Q}_{h}&  \mathrm{P}_{ h}   \end{array}\right],
\end{equation}
where $\mathrm{P}_{h}$,  $\mathrm{Q}_{ h}$ and $ \mathrm{E}_{ h}$ are $d$ dimensional diagonal matrices satisfying 
$$
\mathrm{P}_{h}^2\,-\, (\mathrm{E}_{h}\, +\, \Sigma^{-1}  \, \mathrm{Q}_{h})\, ( \mathrm{E}_{h}\, -\, \Sigma \,\, \mathrm{Q}_{h})\, =\,I_d.    
$$

\noindent Similary to the univariate case, the  error in energy at each step is given by 
\begin{equation}\label{eq:energy_error_2d}
    \begin{array}{c}
     \Delta_{2d}^{(n, h)}\,=\, H(\mathbf{Y}_{n+1})\, - H(\mathbf{Y}_n)\,=\, 
        \dfrac{1}{2} \, \mathbf{Y}_n^T \,\left({\mathcal{K}_{2d}^{(h)}\, }^{T}\,\mathcal{K}_{2d}^{(h)}\, \, -\, \mathcal{D}_{2d}^{-1}\,\right)  \mathbf{Y}_n
    \end{array}
\end{equation}
where $\mathcal{K}_{2d}^{(h)}\, :=\, \mathcal{D}_{2d}^{-1/2}\,{\mathcal{M}_{2d}^{( h, \sigma)}} \, =\left(\begin{array}{cc}  D_{\alpha}^{-1/2}\, \mathrm{P}_{ {h}}   & D_{\alpha}^{-1/2} \mathrm{E}_{ h} + D_{\beta}^{-1/2} \,\mathrm{Q}_{ {h}}     \\\\  D_{\beta}^{-1/2} \mathrm{E}_{ h}  - D_{\alpha}^{-1/2} \,\mathrm{Q}_{ {h}}  &  D_{\beta}^{-1/2}\, \mathrm{P}_{ {h}}   \end{array}\right).$

When $\mathrm{E}_{ h}= \mathbf{0}_{d \times d}$ then the matrix ${\mathcal{E}_{ 2d}^{(h)}} \, :=\,{\mathcal{K}_{2d}^{(h)}}^{T}\,\mathcal{K}_{2d}^{(h)} \,-\, \mathcal{D}_{2d}^{-1}$ is given by 
\begin{equation}\label{eq:matcalEsigmahb}
{\mathcal{E}_{2d}^{(h)}\,} :=
\left(\begin{array}{cc} 
\mathrm{P}_{ h}\, D_{\alpha}\,\mathrm{P}_{ h}
 +  \mathrm{Q}_{ h}\, D_{\alpha}\,\mathrm{Q}_{  h}\, -\, D_\alpha& \mathbf{0}_d\\\\
\mathbf{0}_d &  \mathrm{P}_{ h}\, D_\beta\,\mathrm{P}_{ h}
 + \mathrm{Q}_{  h}\,  D_\beta\,\mathrm{Q}_{ h}\, -\,   D_\beta
 \end{array}\right).
    \end{equation}
 \noindent Since $\mathrm{P}_{ h}$, $\mathrm{Q}_{ h}$, $ D_\alpha$ and $ D_\beta$ are diagonal matrices and $\mathrm{P}^2_{ h}\, +\, \mathrm{Q}^2_{ h}\, = \mathrm{I}_{d}$, then ${\mathcal{E}_{2d}^{(h)}\,} =\, \mathbf{0}_{2d \times 2d}$.


\section{Splitting methods}\label{sec:4}
\noindent  Usually, the  map  implemented within a HMC algorithm is  the  St\"ormer-Verlet method which lies in the class of symmetric splitting methods.
There are several attempts in literature \cite{blanes2021symmetrically},\cite{calvo2019hmc},\cite{baker2018optimal},\cite{baker2019optimal} to introduce more accurate maps in the same class, where {\it accuracy} refers to the performance of the map within the HCM algorithm rather than to the accuracy in approximating the dynamical flow.
Symmetric splitting methods are based on the splitting of the  flow in two (or more) semiflows and is built as a symmetric composition  of   semiflows. When applied to  Hamiltonian dynamics, the  semiflows are themselves Hamiltonian flows, so that they are volume-preserving and reversible maps. The composition of volume-preserving maps results in a volume-preserving map; moreover, as the semiflows are reversible and the composition is symmetric, the splitting map  results reversible (for a detailed proof see \cite{blanes2014numerical}).

\subsection{St\"ormer-Verlet method}\label{sec:4.1}
\noindent The  St\"ormer-Verlet method  is based on the splitting of the  flow in two  semiflows and is built as a symmetric composition  of   semiflows. 
\noindent Setting  $\mathbf{Y}=[\mathbf{Q},\,\mathbf{P}]^T \in \mathbb{R}^{2d}$, it can be useful to denote the Hamiltonian dynamics (\ref{eq:HS}) in vector form as 
$$
\dfrac{d\mathbf{Y}}{dt}\, =\, 
f(\mathbf{Y}):= \left[D^{-1}_{\beta} \mathbf{P},\,-\nabla_{\mathbf{Q}}\,U(\mathbf{Q})\right]^T. 
$$
Let  $\varphi^{[\mathbf{P}]}_t$ and  $\varphi^{[\mathbf{Q}]}_t$ represent the exact flows associated to the dynamics $\dfrac{d\mathbf{Y}}{dt} \,=\,f^{[\mathbf{P}]}(\mathbf{Y})$ and $\dfrac{d\mathbf{Y}}{dt}\,=\,f^{[\mathbf{Q}]}(\mathbf{Y})$, where $f\,= f^{[\mathbf{P}]}\,+\, f^{[\mathbf{Q}]}$ and  $$f^{[\mathbf{P}]}(\mathbf{Y}):=[D^{-1}_{\beta} \mathbf{P},\,\mathbf{0}_d]^T,\qquad f^{[\mathbf{Q}]}(\mathbf{Y}):= [\mathbf{0}_d,\,-\nabla_{\mathbf{Q}}\,U(\mathbf{Q})]^T.$$
The map $\mathbf{Y}_{n+1}=\Psi_h^{(SV)}(\mathbf{Y}_n)$, with
\begin{equation}\label{eq:stormer-verlet}
 \Psi_h^{(SV)}:=\varphi^{[\mathbf{Q}]}_{h/2} \circ
 \varphi^{[\mathbf{P}]}_{h} \circ 
  \varphi^{[\mathbf{Q}]}_{h/2},
\end{equation}
defines the  ({\it velocity}) St\"ormer-Verlet method.\footnote{We mention that the {\it position} St\"ormer-Verlet method starts the integration by solving the semiflow $f^{[P]}$ so that $ \Psi_h^{(SV)}:=\varphi^{[P]}_{h/2} \circ
 \varphi^{[Q]}_{h} \circ 
  \varphi^{[P]}_{h/2}$.}
\bigskip

\noindent A St\"ormer-Verlet step applied to the linear test problem (\ref{eq:arm_general})  will be a linear map, represented in matrix form as $\mathbf{Y}_{n+1}\, =\, \mathcal{M}_2^{( h, \sigma)}\, \mathbf{Y}_{n}$ where $\mathcal{M}_2^{( h, \sigma)}$ is given in (\ref{eq:tildeM}) and 
$$
\mathrm{p}_{ h}\, =\, 1\,-\, \dfrac{ h_\sigma^2}{2},\quad \mathrm{e}_{h}\, = \, \dfrac{\sigma \,{h_\sigma}^3}{4\, (\sigma^2\,+\,1)},\quad  \mathrm{q}_{ h}\, =\,  h_\sigma\,-\,\sigma \, \mathrm{e}_{ h}
$$
with $h_\sigma\,=\, \dfrac{h}{\alpha\, \beta}.$
For $h< 2 \, \alpha \, \beta$  it results   $h_\sigma \,\leq\, 2$, and then the  trajectories are stable as it results  $\mathrm{p}_{ h}^2 \,-\, 1\, <\, 0$.
\bigskip

\noindent Since $\mathrm{e}_{ h}\,\neq \, 0$, the St\"ormer-Verlet integrator cannot preserve the energy when applied to the linear test model (\ref{eq:arm_general}). From Theorem \ref{eq:expectation1} the expectation of the random variable $\Delta_2^{(N)}$ is given by 
$$
\mathbb{E}(\Delta_2^{(N)})\,=\,\dfrac{N}{2}\, \left(\sigma\,+\, \dfrac{1}{\sigma}\right)^2\, \left(\dfrac{\sigma \,{h_\sigma}^3}{4\, (\sigma^2\,+\,1)}\right)^2\, =\, \dfrac{N}{32}\,  h_\sigma^6 \, = \, T^*\,\left(\frac{ h_\sigma}{2}\right)^5.
$$
\subsection{One-parameter  family of second order splitting methos}\label{sec:4.2}
\noindent Different improvements of the St\"ormer-Verlet method can be found in literature. Often the idea  is to tune  some free parameter in some suitable class of methods in order to maximizing, for the linear test model, the length of the stability interval, subject to the annihilation of some error constants as in  \cite{predescu2012computationally} or to ensure good conservation
of energy properties in linear problems so reducing the energy error as in  \cite{blanes2014numerical}.
In both cases the aim is to suggest methods able to  increase the number of accepted proposals in the  HCM algorithm with respect to the   St\"ormer-Verlet method.

\bigskip

\noindent In this paper a new criterion is adopted for tuning the free parameter $b \in \mathbb{R}$ in the class of second order splitting methods, called nSP2S (new splitting two step method)  
\begin{equation}\label{composition_2old}
 \Psi_h^{(b)}:=\varphi^{[\mathbf{Q}]}_{b \,h} \circ
 \varphi^{[\mathbf{P}]}_{h/2} \circ 
 \varphi^{[\mathbf{Q}]}_{(1-2\,b)h} \circ
 \varphi^{[\mathbf{P}]}_{h/2} \circ \varphi^{[\mathbf{Q}]}_{b \, h}.
\end{equation}
 The aim is to exactly preserve the energy so to have all proposals accepted when HCM is applied to Gaussian distributions.  

\noindent Let us  underline that this idea is not novel in the field of numerical approximation of Hamiltonian dynamics \cite{pace2015splitting}. However, the benefits of this approach have not been analyzed in the field of  Hamiltonian Monte Carlo algorithms. Before proceeding, as observed in \cite{blanes2014numerical}, notice that the  mapping $\Psi_h^{(b)}$ in (\ref{composition_2old}) is volume-preserving, reversible and symplectic. Moreover, we will consider $b\,\neq\, 0,\, 1/2$  as the method reduces to the classical velocity and position St\"ormer-Verlet integrators in these cases.

\bigskip

\noindent The class of second order methods $\mathbf{Y}_{n+1}=\Psi_h^{(b)}(\mathbf{Y}_n)$,  with $\Psi_h^{(b)}$ given in (\ref{composition_2old}) when applied to the model test system (\ref{eq:arm_general}) can be written as  a linear map, represented in matrix form as $\mathbf{Y}_{n+1}\, =\,\mathcal{M}_2^{( h, \sigma)}\,  \mathbf{Y}_{n}$ where $\mathcal{M}_2^{( h, \sigma)}\,  $ is given in (\ref{eq:tildeM}) and 
$$
\begin{array}{l}
  \mathrm{p}_{ h}\, =\, 1\,-\, \dfrac{ h_\sigma^2}{2}\,+\, \dfrac{h_\sigma^4}{4}\, b\,(1\,-\,2\,b) ,  \\\\
  \mathrm{q}_{ h}\, =\, \dfrac{b^2(1-2\,b)}{4\,\left(\sigma ^2+1\right)}\,{{h_\sigma }}^5+\dfrac{4\,b^2+2\,b\,\sigma ^2-4\,b-\sigma ^2}{4\left(\sigma ^2+1\right)}\,{{h_\sigma }}^3\,+\,{h_\sigma}.
\end{array}
$$
and 
\begin{equation}\label{eq:econh}
  \mathrm{e}_{ h}\, =  \mathrm{e}_{ h}\,(b) \, =\,\dfrac{h_\sigma^3 \sigma}{4(\sigma^2+1)} \left(2\,b^3\,{{h_\sigma }}^2-b^2\,{{h_\sigma}}^2-4\,b^2+6\,b-1\right),
\end{equation}
where, as before, $ h_\sigma\,:=\, \dfrac{h}{\alpha\, \beta}.$

\noindent  The stability interval can be deduced from the known result given in \cite{blanes2014numerical} i.e. 
 \begin{equation}\label{eq:stability}
 0\,<\, h_\sigma=\dfrac{h}{\alpha\, \beta} \, <\, \min\left\{ \,\sqrt{\dfrac{2}{b}}, \, \sqrt{\dfrac{2}{1/2\,-\,b}}\,\right\},  \qquad 0\, < \, b\, < \dfrac{1}{2}.
 \end{equation}

\noindent The application of Theorem \ref{eq:expectation1} gives the expectation of the random variable  $\Delta_2^{(N)}$ 
$$
\mathbb{E}(\Delta_2^{(N)})\,=\, T^*\,\left(\frac{ h_\sigma}{2}\right)^5 \left(2\,b^3\,{{h_\sigma }}^2-b^2\,{{h_\sigma}}^2-4\,b^2+6\,b-1\right)^2
$$
which can be nullified exploiting the following result which generalizes Theorem $1$ given  in \cite{pace2015splitting}.
\begin{theorem}\label{Th1}
For all $b,\,h >0$ define  
\begin{equation}\label{eq:rdib}
R(b,h)\, :=2\,\left(\dfrac{h}{\alpha\, \beta}\right)^2\,b^3\,-\,\left(4+ \left(\dfrac{h}{\alpha\, \beta}\right)^2\right)\,b^2\,+\,6\,b\,-\,1.
\end{equation}
Fix $h>0$ and consider $b_{h}$ a real root of the  third degree polynomial (\ref{eq:rdib}) in the variable $b$; 
then the scheme $\mathbf{Y}_{n+1}=\Psi_h^{(b_h)}(\mathbf{Y}_n)$,  with $\Psi_h^{(b_h)}$ given in (\ref{composition_2old}) is  energy-preserving  for the test model (\ref{eq:arm_general}).
\end{theorem}
\begin{proof}
Write  $\mathrm{e}_{h}(b)$   in (\ref{eq:econh}) as  $\mathrm{e}_{ h}(b) \,=\, \dfrac{ \sigma}{4(\sigma^2+1)} \,\left(\dfrac{h}{\alpha\, \beta}\right)^3\, R(b, h)$; then, from $R(b_h, h)=0$ it follows   $\mathrm{e}_{ h}(b_h)\, =\, 0$. From Theorem  \ref{thm:univariatecase}, the result follows.
\end{proof}

\bigskip

\noindent In the HMC framework it can be more useful to adopt a different  perspective:

\begin{theorem}\label{Th2}
Let $\dfrac{3-\sqrt{5}}{4}\, < \, b\, \leq \dfrac{1}{4}$ and consider  
 \begin{equation}\label{hr}
 h_b\, :=\,  \sqrt{\dfrac{4\, b^2\,-\, 6\, b\, +\, 1}{b^2\, (2\,b\, -\, 1)}}.
\end{equation}
Then the scheme in (\ref{composition_2old}) given by  $\mathbf{Y}_{n+1}=\Psi_{h}^{(b)}(\mathbf{Y}_n)$ with  $h:={\alpha\, \beta\,  h_b}$  provides a stable energy-preserving approximation of the test model (\ref{eq:arm_general}).
\end{theorem}
\begin{proof}
 Consider  $R(b,h)$ in (\ref{eq:rdib}) as a second order polynomial with respect to $ h_\sigma = \dfrac{h}{\alpha\, \beta}$ which admits the positive root $h_b$
given in (\ref{hr}) for  $\dfrac{3-\sqrt{5}}{4}\, < \, b\, < \dfrac{1}{2}$; as a consequence $$\mathrm{e}_{h}(b)\, =\, \dfrac{ \sigma }{4(\sigma^2+1)} \,\left(\dfrac{h}{\alpha\, \beta}\right)^3 R(b, h)=  \dfrac{ \sigma  h_b^3}{4(\sigma^2+1)} \, (2\,h_b^2\,b^3\,-\, (4+  h_b^2)\,b^2+6b-1)=0 $$ and, from Theorem  \ref{thm:univariatecase}, the conservation of energy follows. Moreover, under the hypothesis of $b$ bounded by $\dfrac{1}{4}$  from above, it results $$  h_b  \, \leq\, \sqrt{\dfrac{4}{1\,-\,2b}}\, = \min\left\{ \,\sqrt{\dfrac{2}{b}}, \, \sqrt{\dfrac{2}{1/2\,-\,b}}\,\right\}$$
so that $h= \alpha \, \beta \, h_b$ satisfies the stability condition
(\ref{eq:stability}).
\end{proof}

\noindent  An important consequence which will be useful  to extend the described result to the multivariate case, is the following. 
\begin{theorem}\label{corollary}
With the notations used above,  the scheme  $\mathbf{Y}_{n+1}=\tilde \Psi_{  h_b}^{(b)}(\mathbf{Y}_n)$  with \begin{equation}\label{composition_2tilde}
 \tilde \Psi_{ h_b}^{(b)}:=\tilde \varphi^{[Q]}_{b  h_b} \circ
\tilde  \varphi^{[P]}_{ h_b/2} \circ 
\tilde  \varphi^{[Q]}_{(1-2\,b)   h_b} \circ
\tilde  \varphi^{[P]}_{ h_b/2} \circ \tilde \varphi^{[Q]}_{b   h_b}
\end{equation}
where 
 $\tilde \varphi^{[P]}_t$ and  $\tilde \varphi^{[Q]}_t$ represent the exact flows of the dynamics $\dfrac{d\mathbf{Y}}{dt} =[\sigma^{-1} P,\,\mathbf{0}]^T$ and $\dfrac{d\mathbf{Y}}{dt}=[\mathbf{0},\,-\sigma \,Q]^T$, respectively,  provides a stable energy-preserving approximation of the test model (\ref{eq:arm_general}). 

\end{theorem}
\begin{proof}
It is enough to observe that the scheme $\mathbf{Y}_{n+1}=\tilde \Psi_{  h_b}^{(b)}(\mathbf{Y}_n)$
is equivalent to the scheme (\ref{composition_2old}) given by  $\mathbf{Y}_{n+1}=\Psi_{h}^{(b)}(\mathbf{Y}_n)$ with $h\, = \alpha\, \beta\,   h_b$.
\end{proof}

\subsection{Generalization to multivariate Gaussian distributions}\label{sec:4.3}

\noindent In Theorem \ref{corollary} it was shown how to  build a symplectic, reversible, energy-preserving scheme for the $j$th   oscillator (\ref{eq:arm_generalmulti}), for $j=1,\dots,d$. Setting $\mathbf{Y}^{(j)}:=[Q_j,\, P_j],$ consider the scheme  $\mathbf{Y}^{(j)}_{n+1}=\tilde \Psi_{  h_b}^{(b)}(\mathbf{Y}^{(j)}_n)$ with \begin{equation}\label{composition_2jesima}
 \tilde \Psi_{  h_b}^{(b)}:=\tilde \varphi^{[Q_j]}_{b  h_b} \circ
\tilde  \varphi^{[P_j]}_{  h_b/2} \circ 
\tilde  \varphi^{[Q_j]}_{(1-2\,b)  h_b} \circ
\tilde  \varphi^{[P_j]}_{  h_b/2} \circ \tilde \varphi^{[Q_j]}_{b  h_b}
\end{equation}
where 
 $\tilde \varphi^{[P_j]}_t$ and  $\tilde \varphi^{[Q_j]}_t$ represent the exact flows of  $$\dfrac{d\mathbf{Y}^{(j)}}{dt} \,=\,[\sigma_j^{-1} \, P_j,\,0]^T, \qquad  \dfrac{d\mathbf{Y}^{(j)}}{dt}\,=\,[0,\,-\sigma_j \,  Q_j]^T, \quad \sigma_j\, =\, \dfrac{\beta_j}{\alpha_j},$$ for $j=1,\dots,d$. It is a symplectic, reversible, stable scheme for   the $j$  oscillator (\ref{eq:arm_generalmulti}),  which preserves the $j$-th Hamiltonian  $H_j(Q_j,P_j)\,=\, \dfrac{1}{2} \left(\dfrac{Q_j^2}{\alpha_j^2}\, +\, \dfrac{P_j^2}{\beta_j^2}\right)  =\, \dfrac{1}{2\, \alpha_j\, \beta_j} \left(\sigma_j\, Q_j^2\,+\,  \dfrac{P_j^2}{\sigma_j}\right).$

\noindent Now we are searching for  symplectic, reversible, energy-preserving schemes for the  $d$-dimensional test model (\ref{arm_osc_multgen}). With the same notations adopted in Section (\ref{sec:3.2}), i.e. $D_{\alpha}$ and $D_{\beta}$ are $d\, \times d$ diagonal matrices with entries $\alpha_j^2$ and $\beta_j^2$, respectively, for $j=1,\dots,\, d$ 
 and  $\Sigma\, = \, D^{1/2}_{\beta}\, D^{-1/2}_{\alpha}$, we can give the following result


\begin{theorem}\label{th:multi}
For $\dfrac{3-\sqrt{5}}{4}\, < \, b\, \leq \dfrac{1}{4}$ the method $\mathbf{Y}_{n+1}\, =\,\tilde \Psi_{ h_b}^{(b)}\, \mathbf{Y}_{n}$ with  $ h_b$ defined in (\ref{hr}) and, with abuse of notations,   \begin{equation}\label{composition_3}
 \tilde  \Psi_{  h_b}^{(b)}:=\tilde \varphi^{[\mathbf{Q}]}_{b \, h_b} \circ
\tilde  \varphi^{[\mathbf{P}]}_{ h_b/2} \circ 
\tilde  \varphi^{[\mathbf{Q}]}_{(1-2\,b)  h_b} \circ
\tilde  \varphi^{[\mathbf{P}]}_{ h_b/2} \circ \tilde \varphi^{[\mathbf{Q}]}_{b \,  h_b}, 
\end{equation}
where
 $\tilde \varphi^{[\mathbf{P}]}_t$ and  $\tilde \varphi^{[\mathbf{Q}]}_t$ represent the exact flows of  $$\dfrac{d\mathbf{Y}}{dt} \,=\,[\Sigma^{-1} \, \mathbf{P},\,\mathbf{0}_d]^T, \qquad \dfrac{d\mathbf{Y}}{dt}\,=\,[\mathbf{0}_d,\,-\Sigma \,  \mathbf{Q}]^T,$$  provides a  symplectic, reversible, stable  approximation for the system (\ref{arm_osc_multgen}),  which preserves the Hamiltonian (\ref{eq:Hmulti}).
\end{theorem}
\begin{proof}

The method (\ref{composition_3})    can be expressed  as $\mathbf{Y}_{n+1}\, =\, \mathcal{M}_{2d}^{(  h_b, \Sigma)}\, \mathbf{Y}_{n}$ where $\mathcal{M}_{2d}^{(  h, \Sigma)}$ is given in (\ref{eq:tildeMmulti})
and 
$$
\begin{array}{l}
  \mathrm{P}_{ h_b}\, =\, \left(1\,-\, \dfrac{  h_b^2}{2}\,+\, \dfrac{  h_b^4}{4}\, b\,(1\,-\,2\,b)\right) \,   \mathrm{I}_{d}, \\\\
  \mathrm{Q}_{  h_b}(j,j) = \dfrac{\sigma_j^2\,b^2(1-2\,b)}{4\,\left(\sigma_j ^2+1\right)}\,{{  h_b }}^5+\dfrac{4\,\sigma_j^2 b^2+2\,b\,-4\,b\, \sigma_j ^2 -1}{4\left(\sigma_j ^2+1\right)}\,{{  h_b }}^3\,+\,{  h_b}, \\\\
   \mathrm{E}_{ h_b}(j,j)   =\,\dfrac{  h_b^3 \sigma_j}{4(\sigma_j^2+1)} \left(2\,b^3\,{{  h_b }}^2-b^2\,{{ h_b}}^2-4\,b^2+6\,b-1\right)\, =\, 0,
\end{array}
$$
and $  \mathrm{Q}_{  h_b}\,(i,j)\, =\, \mathrm{E}_{  h_b}(i,j)\,=\, 0\,$  for $i\neq j, \, i,\,j=1,\dots,d.$ As $\mathrm{E}_{ h_b}= \mathbf{0}_{d \times d}$  then, from  (\ref{eq:matcalEsigmahb}),  
${\mathcal{E}_{ 2d}^{(h_b)}}\, =\, \mathbf{0}_{2d \times 2d }$ and the energy error in (\ref{eq:energy_error_2d})  nullifies i.e.  $\Delta_{2d}^{(n, h)}=0$.
\end{proof}

\section{Step size selection    for sampling from generic distributions}\label{sec:5}
In this section we propose a criterium for select the step size $h$ within HMC processes  for sampling from generic distributions by means of the splitting method (\ref{composition_2old}). It relies on the following preliminar  result 
\begin{theorem}\label{th:multi_generic_new}
 For $\dfrac{3-\sqrt{5}}{4}\, < \, b\, \leq \dfrac{1}{4}$, the method (\ref{composition_2old}) with  $h= h_b$ defined in (\ref{hr})   provides a  symplectic, reversible, stable  approximation for the system (\ref{eq:HS}),  which preserves the Hamiltonian (\ref{eq:Hami_general}) whenever $U(\mathbf{Q)}=\dfrac{1}{2} \, \mathbf{Q}^T \, {D}_{\beta}^{-1}\, \mathbf{Q}$.
\end{theorem} 
\begin{proof}
It is enough to notice that, for $U(\mathbf{Q)}=\dfrac{1}{2} \, \mathbf{Q}^T \, {D}_{\beta}^{-1}\, \mathbf{Q}$,  $D_{\alpha}\equiv D_{\beta}$ and  $\Sigma\, = \, \, D^{1/2}_{\beta}\, D^{-1/2}_{\beta} = I_d$; then the algorithm (\ref{composition_2old}) reduces to (\ref{composition_3}).
\end{proof}
\bigskip It is worth observing that $\Sigma\,=\, I_d$ also when  $D_{\alpha}= D_{\beta}=I_d$. This means that, whenever  $U(\mathbf{Q)}=\dfrac{1}{2} \, \mathbf{Q}^T \,  \mathbf{Q}$ we will associate as kinetic the function $U(\mathbf{P)}=\dfrac{1}{2} \, \mathbf{P}^T \,  \mathbf{P}$.

\bigskip
\noindent Hence, for sampling from generic distributions within a HMC processes, we propose to  replace the classical St\"ormer-Verlet algorithm in (\ref{eq:stormer-verlet}) with the second order splitting method defined 
in (\ref{composition_2old})  and to adopt the step size selection $h=h_b$ defined in (\ref{hr}). The rationale is  that, differently from the St\"ormer-Verlet method,   for both univariate and multivariate Gaussian test problems, the one-parameter map (\ref{composition_2old})  can advance with a  suitable  step size which nullifies the energy error allowing  all proposals to be accepted as in the {theoretical} HMC \cref{alg:buildtree}. 

\subsection{Adaptive selection of the $b$ parameter}\label{b_selection}
Each value of the  $b$ parameter in the interval $]b_{min},\, b_{max}]:=\left]\frac{3-\sqrt{5}}{4},\,\frac{1}{4}\right]$ detects a specific method in the class of the splitting methods (\ref{composition_2old}). Hence, we may wonder about what is the 'best' choice and, consequently, the 'best method' to adopt.  Classical criteria might be followed:

\bigskip
\begin{enumerate}

\item choose $b=b_{max}=0.25$ to enlarge $h_b$ as much as  possible. Consequently, $h_{b_{max}}\approx 2.828$. This is a very large step which can  be used, in practice, only for Gaussian distributions and for very low-dimensional non-stiff problems;\\

\item set $b$ at the value $b_{BCS}=\dfrac{3-\sqrt{3}}{6}$ indicated as optimal in \cite{blanes2014numerical}. The resulting step is  $ h_{b_{BCS}} \approx 1.8612$;\\

\item enlarge $h_b$ as much as  possible increasing $b$ but taking into account that stability decreases when we approach the roots of   $\mathrm{p}_{h_b}^2 \,-\, 1\, =\, 0$. The best choice corresponds to $b$ such that  $\mathrm{p}_{ h_b}^2\, =\,0$.  We find $b = b_{stab}\approx 0.2008$ and $  h_{b_{stab}}\, \approx 1.3432$;\\

\item  choose   $b$ in order to minimize  the leading error term $ k^2_{3,1}\, + k^2_{3,2}$ with 
$
k_{3,1}\,= \dfrac{12\, b^2 -12 \, b\,+2}{24}$ and  $k_{3,2}\,= \dfrac{-6\,b\,+\,1}{24}$. In this case the optimal choice corresponds to $b=b_{ML}\approx 0.1932$ (see  \cite{mclachlan1995numerical}) and the resulting step is $ h_{b_{ML}} \approx 0.6549$;

\end{enumerate}

\bigskip 
\noindent In our implementations, the choice of the parameter $b$ is initially finalized to make fair comparisons with existing schemes. This means that we set $b$ to the values which correspond to  step size $h=h_b$ varying in the same numerical range considered in benchmark tests.

\noindent However, a  promising strategy we are going to propose is  an 'adaptive' choice of the method. Starting from one of the classical choices of $b=b_{init}$  as described above, we decrease this value of a fixed percentage each time a sample is not accepted. Since the allowed maximum value of $b=b_{max}$ provides the maximum step size $h_{b_{max}}\approx 2.828$, then $T^*$ is chosen larger that $3$ in order to have $N\geq 1$. The HMC algorithm with the  initial $b=b_{max}$ selection with $75 \% $ of reduction is described in \cref{alg:buildtreenovel}.   In our simulations, we will also test the performance of the proposed integrator built on the proposed adaptive  approach.

\begin{algorithm}
\caption{Novel HMC algorithm }
\label{alg:buildtreenovel}
\begin{algorithmic}
\STATE{Draw $\mathbf{q}^{(1)} \sim \pi(\mathbf{q})$,  $\mathbf{q}^{(1)}\in \mathbb{R}^d$, $L\geq 1$,  $b_{init}=b_{max}, red=75\%$, set  $i=0$}
\WHILE{$i < L$}
\STATE{i=i+1}
\STATE{Draw  $\mathbf{p}^{(i)} \sim \mathcal{N}(0,D_{\beta})$}
\STATE{Set $factor\,=\,b_{init}-b_{min}$}
\STATE{Set $(\mathbf{Q}_0, \,\mathbf{P}_0 ) =  (\mathbf{q}^{(i)},\, \mathbf{p}^{(i)})$}, set $j\,=\,0$
\WHILE{$j < 1$}
\STATE {Randomly  choose $T^*\geq 3$}.
\STATE {Set $b=b_{min}\,+\,factor$ and $h=\sqrt{\dfrac{4\, b^2\,-\, 6\, b\, +\, 1}{b^2\, (2\,b\, -\, 1)}}$ }
\STATE{Set $N\geq 1$ such that $T^*\, =\, N\, h$} 
\STATE{Evaluate $(\mathbf{Q}_{n+1}\, \mathbf{P}_{n+1})\, =\,  \Psi_{h}^{(b)}(\mathbf{Q}_{n},\, \mathbf{P}_{n}) $, \,for $n=0,\dots N-1$}
\STATE{if $\left(\mathbf{Q}_N,\, \mathbf{P}_N\right)\, \neq \, (\mathbf{Q}_0, \,\mathbf{P}_0 )$, j\,=\,1 }
\ENDWHILE
\STATE{Set $\left(\mathbf{q}^*,\, \mathbf{p}^*\right)\,=\, \left(\mathbf{Q}_N,\, \mathbf{P}_N\right)$ }
\STATE{Calculate 
$\alpha = \text{min}\left(1,\exp\left(H(\mathbf{q}^{(i)},\mathbf{p}^{(i)}) -H(\mathbf{q}^*,\mathbf{p}^*)\right)\right)$}
\STATE{Draw  $u \sim \mathcal{U}(0,1)$}
\STATE{Update: if $\alpha >u$ then $\mathbf{q}^{(i+1)} = \mathbf{q}^*$;\\ otherwise $\mathbf{q}^{(i+1)} = \mathbf{q}^{(i)}$, \, $factor\,=\,red \cdot\, factor$;}
\ENDWHILE
\RETURN  Markov chain $\mathbf{q}^{(1)}, \, \mathbf{q}^{(2)}, \dots, \, \mathbf{q}^{(L)}$
\end{algorithmic}
\end{algorithm}
\bigskip
\section{Numerical examples}\label{sec:6}
\subsection{Bivariate Gaussian distributions}\label{sec:6.1}
As first example, the  simple $d=2$ dimensional test  in \cite{neal2011mcmc} is proposed, in order to numerically show the energy-preserving property of the proposed splitting technique.  Consider sampling two position variables $\mathbf{X}\, =[X_1\, X_2]^T$ from a bivariate Gaussian distribution with  zero means, unit standard deviation  and covariance $0.95$. Two corresponding momentum variables $\mathbf{P}\, =[P_1,\, P_2]$ defined to have a Gaussian distribution with unitary covariance matrix, are introduced. We then define the Hamiltonian as  $$U(\mathbf{X})\, +\, K(\mathbf{P})\, =\, \dfrac{1}{2} \, \mathbf{X}^T \,  S_{95}^{-1}\, \mathbf{X} + \dfrac{1}{2} \, \mathbf{P}^T \, \mathbf{P}, \qquad S_{95}= \left(\begin{array}{cc}
    1 & 0.95 \\
    0.95 & 1
\end{array}\right).$$ In order to describe the problem with notations suitable for the application of the proposed procedure, we diagonalize the symmetric  matrix 
$
S_{95}\, =\, V^T\, D_{\alpha}\, V  \, 
$
with $V$ unitary matrix of eigenvectors. In doing so, the Hamiltonian can be written as 
$$U(\mathbf{Q})\, +\, K(\mathbf{P})\, =\, \dfrac{1}{2} \, \mathbf{Q}^T \,  D_{\alpha}^{-1}\, \mathbf{Q} + \dfrac{1}{2} \, \mathbf{P}^T \, \mathbf{P}, \qquad D_{\alpha}= \left(\begin{array}{cc}
    0.05 & 0 \\
    0 & 1.95
\end{array}\right)$$ 
with $\mathbf{Q}\, =\, V\, \mathbf{X}$.  To illustrate the basic functionality of the novel  proposed integration method (\ref{composition_3}) applied with $\Sigma = D_{\alpha}^{-1/2}$ (nSP$2$S method), we compare it with the St\"ormer-Verlet method (\ref{eq:stormer-verlet}) (SV-method), and with the two and three step methods presented in \cite{blanes2014numerical}, (hereafter denoted with SP$2$S and SP$3$S). We run the experiment by choosing a path length $T^*=5$ for all the integrators and with a step size $h=0.005$ for the three competitors and with $b=b_{stab}=0.2008$ and $h_{b_{stab}}=1.3432$ of our approach. 
Figure \ref{Esempio2} shows the acceptance rate (AR) of the proposals for all of the methods considered; as  theoretically predicted, despite the very large step size,  nSP$2$S maintains the maximum acceptance rate  $AR=1$. 
\begin{figure}[h]\label{Esempio2}
\centering
  \subfigure[AR=0.6610]{\includegraphics[width=0.45\linewidth]{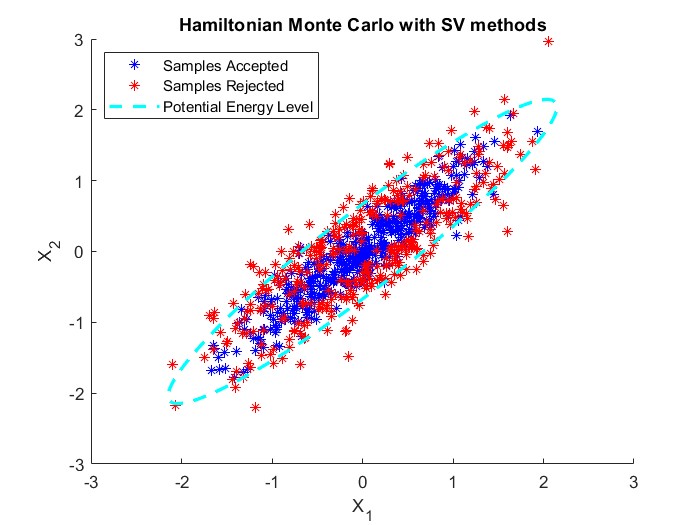}}
  \subfigure[AR=0.842]{\includegraphics[width=0.45\linewidth]{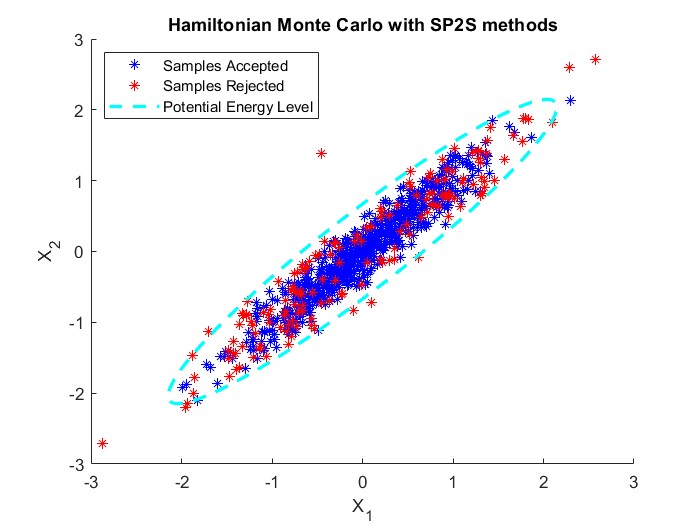}}
 \subfigure[AR=0.882]{\includegraphics[width=0.45\linewidth]{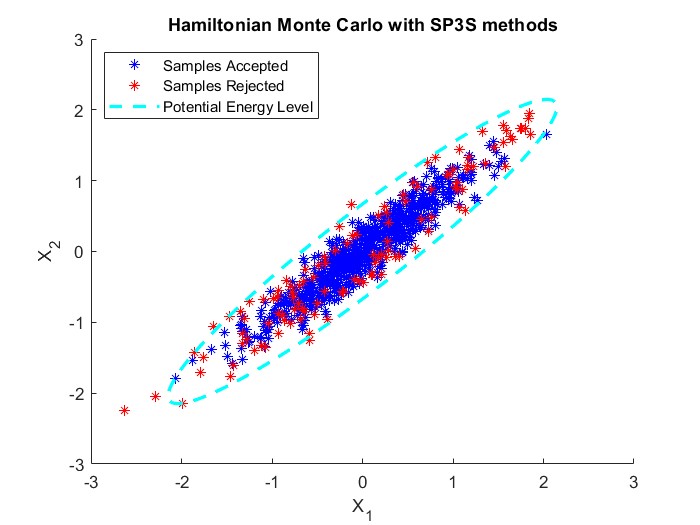}}
  \subfigure[AR=1]{\includegraphics[width=0.45\linewidth]{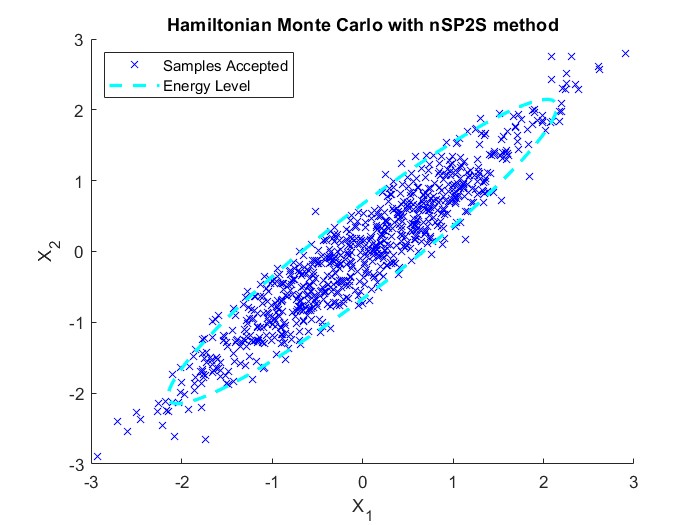}}
 \caption{Accepted and rejected samples of different integrator, for a Bivariate distribution. Iterations of HMC samples with $T^*=5$ and $h=0.005$ for (a), (b), (c), and $h\, = h_{b_{max}}\, \approx 1.3432$ for (d). 
 }
 \end{figure}

\subsection{Multivariate Gaussian distribution}\label{sec:6.2}
Regarding a more general case, we have considered, as a target a multivariate Gaussian distribution as considered in \cite{neal2011mcmc} in the same form adopted in \cite{blanes2014numerical} i.e. 
\[
 \pi(\mathbf{q})\propto \Bigg(-\frac{1}{2}\sum_{j=1}^{d}j^2q^2_{j}\Bigg).
\]
We consider different dimensions, from $d=256$ to $d=1024$ 
with potential energy function  $U(\mathbf{Q})=\dfrac{1}{2} \, \mathbf{Q}^T \,  D_{\alpha}^{-1}\, \mathbf{Q}$ in which the variables are independent, with  zero mean and $D_{\alpha}$ entries given by  standard deviations $\alpha_j\, = {1/j^2}$ for $j=1,\dots, d$.  Kinetic energy function $K(\mathbf{P})\,=\, \dfrac{1}{2} \, \mathbf{P}^T \, \mathbf{P}$ is set  as above.

\noindent For the experiments we compared our  algorithm (\ref{composition_3}) applied with $\Sigma = D_{\alpha}^{-1/2}$ with the family of integrators considered in \cite{calvo2019hmc}. This kind of integrators are \emph{second order accurate} and depend on a real parameter. The authors consider different integrators by varying the value of the parameter (see \cite{calvo2019hmc} for the details of integrators). For our purpose we chose as competitors, only those that in the paper are called LF, and BlCaSa. The first because it corresponds to three consecutive time-steps of the classic St\"ormer-Verlet (or equivalently, LeapFrog) method (\ref{eq:stormer-verlet}), the second because it is the most performing one when considering the case of multivariate Gaussian \cite{calvo2019hmc}. We do not consider here further improvements of the methods as provided in \cite{blanes2021symmetrically}, as, likewise the previous methods,  they do not retain the energy of Gaussian distributions.

\noindent To perform the experiments we used the same parameters used in the paper \cite{calvo2019hmc}. 
\begin{figure}[h]\label{multivariate_conf_acc}
  \centering
  \subfigure[LF]{\includegraphics[width=0.45\linewidth]{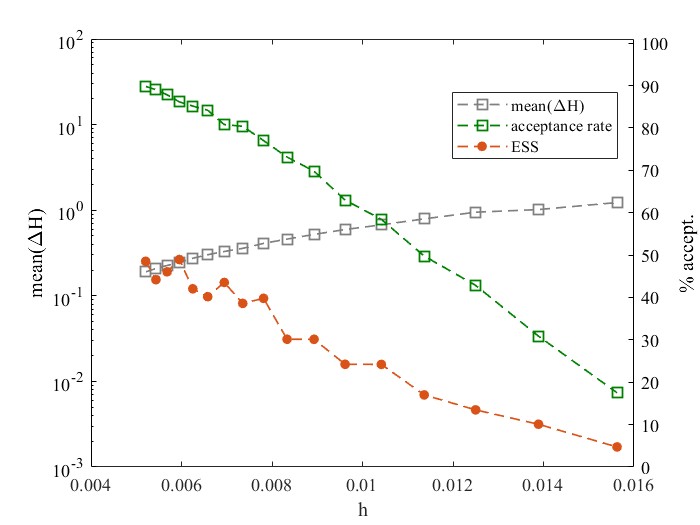}}
  \subfigure[BlCaSa]{\includegraphics[width=0.45\linewidth]{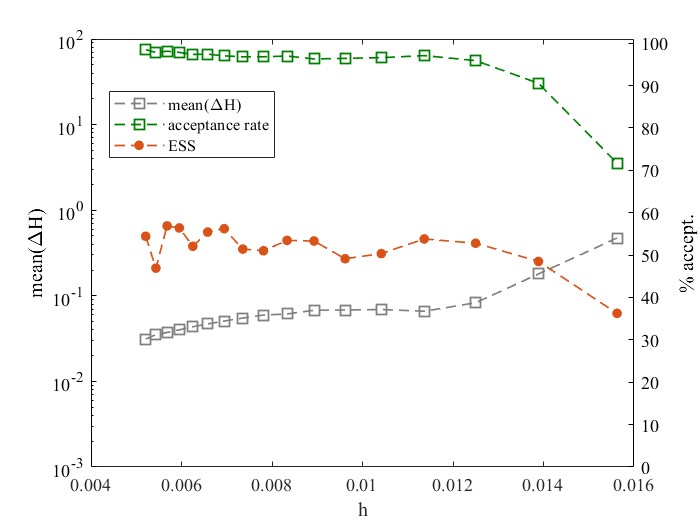}}\\
   \subfigure[nSP2S-1]{\centering\includegraphics[width=0.45\linewidth]{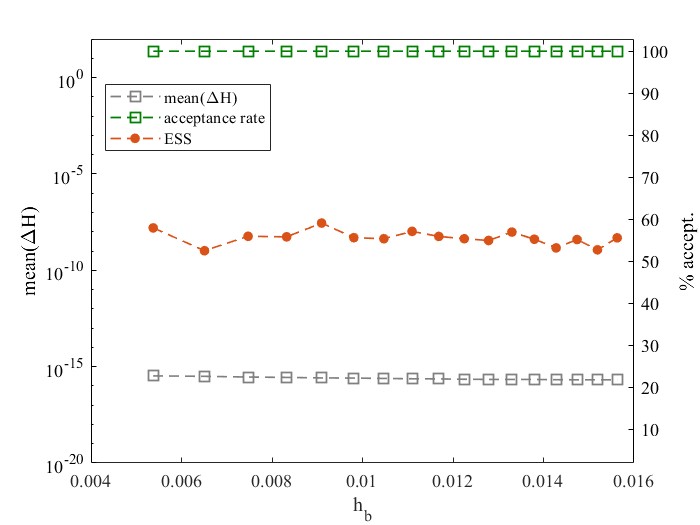}}
   \subfigure[nSP2S-2]{\centering\includegraphics[width=0.45\linewidth]{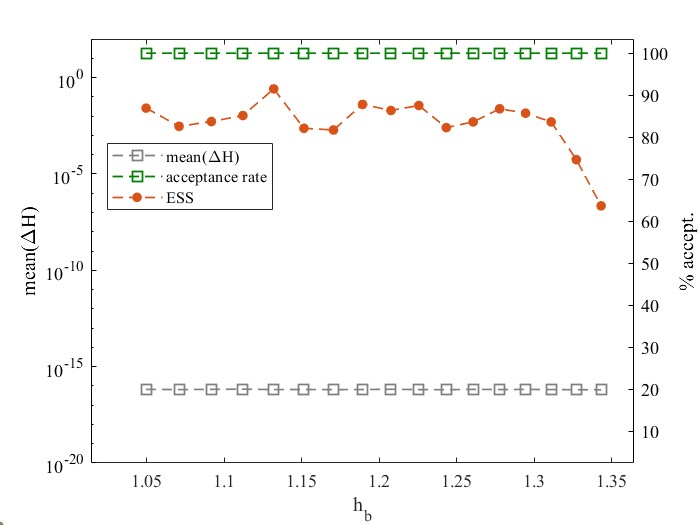}}
 \caption{Multivariate Gaussian target with d=256; acceptance rate percentage (green squares), mean of $\Delta H$ (grey squares) and ESS percentage (orange points) comparison for $LF$, $BlCaSa$ integrators, measuring with different step size $h$ and nSP$2$S-1 and nSP$2$S-2 by varying $h_b$.}
\end{figure}
In particular,  the number of samples, each of dimension $d=256$, has been set to $L=5000$ choosing a number of burn-in samples  equal to $1000$. 
The initial $\mathbf{q}^{(1)}$ is drawn from the target $\pi(\mathbf{q})$.  For LF and for BlCaSa we choose  path length $T^*=5$ and the time-steps $N=320,360,\ldots,960$, with a corresponding step sizes $h =5/320,5/360,\ldots,5/960$ ($h$ then varies between $\approx 1.6\times 10^{-2}$ and $ 5.2\times 10^{-3}$).  For our method, where the step size $h_b$ depends from the parameter $b$, we performed two different experiments. 
\noindent The first, called nSP$2$S-$1$ aims at  making a fair comparison with  LF and for BlCaSa: we choose different values of parameter b in the range $( 0.1909831513, 0.1909842368)$ which make the corresponding step sizes $h_b$,  time-steps $N$ and path length $T^*=5$ equal to those used by competitors.
In the second experiment, called nSP$2$S-$2$, we show the performance of the proposed integrator for  values of the  parameter $b$ in the range $(0.1968,0.2008)$ that provide large step sizes $h_b$ varying between $1.05$ and $1.35$.  In this case the time-interval has been randomized with $\pm 40\%$ variations around  $5$, with this ensuring  that $3<T^*<7$.
For all the experiment  we have measured  the acceptance rate, the mean of  $4000$ samples of the energy errors $\Delta H^{(i)}(\mathbf{q}^{(i)},\mathbf{p}^{(i)})=H(\mathbf{q}^*,\mathbf{p}^*)- H(\mathbf{q}^{(i)},\mathbf{p}^{(i)})$ and the effective sample size ESS of the first component $q_1^{(i)}$ of $\mathbf{q}^{(i)}$  which corresponds to the component with largest standard deviation $\alpha_1 = 1$ (see  \cite{MARTINO2017386}).

\noindent In  Figure \ref{multivariate_conf_acc}  on the left vertical axis,   the means of the energy errors (gray square), are reported on a  logarithmic scale, for the different values of the stepsize $h$.
On the right vertical axis, for the LF  (a), for BlCaSa(b) and for the first experiment with nSP$2$S method (c),  the accepted percentage rate (green square) and the ESS  (orange circle)   are plotted in linear scale. In (d) the same quantities versus  step size $h_b$ are depicted for the second experiment performed with the nSP$2$S method. As expected in both cases (c) and (d), our approach has a $100\%$ acceptance rate with mean of energy errors of order $10^{-16}$. In particular the ESS remains above $60\%$ in both cases. It’s worth noting that the computational cost of the experiment  nSP$2$S-$2$ is much lower than the experiment nSP$2$S-$1$ as a larger step size has been adopted without decrement in terms of performance. 

\noindent In addition, for a reduced number of samples  $L=1000$, we performed  a qualitative comparative analysis among the methods  by  estimating the mean and the standard deviation of the samples.
\begin{figure}[H]\label{mean_sd}
  \centering\includegraphics[width=0.45\linewidth]{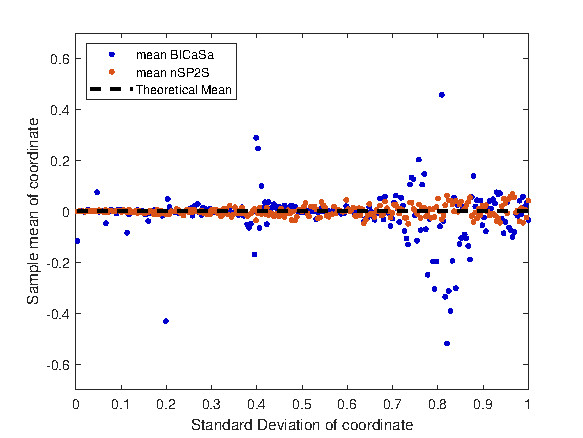}
   \centering\includegraphics[width=0.45\linewidth]{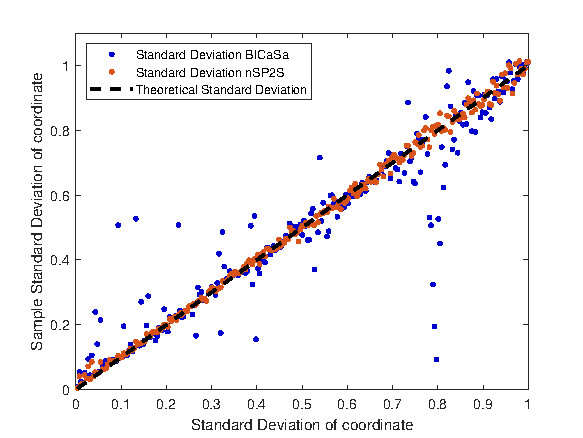}
 \caption{Multivariate distribution. Estimates of means (left) and standard deviations (right) for the $d=256$ dimensional example. On the $x$-axes are reported the standard deviations $\alpha_j$ of each variable $q_j$ for $j=\, 1,\dots, d$, on the $y$-axes the estimated values each evaluated from $L=1000$ iterations, are reported.}
\end{figure}
We run BlCaSa by setting $h=0.014$ and $N=5/h$ and nSP$2$S with $b=0.191$, $h_b=0.0580$ and $N=25$. 
The estimates for the mean and the standard deviations, evaluated as simple means and standard deviations of the values from the $L\, =1000$ iterations for  each of the $d\, =256$ variables,  against the theoretical values of standard deviations  $\alpha_j$, for $j=1,\dots,\, d$, are shown in Figure \ref{mean_sd}. Notice that the error in the estimates for the means and standard deviations obtained with the HMC algorithm with trajectories evaluated with  nSP$2$S appears  even better than the ones provided by HCM with BlCaSa method. Of course, the saving in computational cost was very evident as, for each iteration, only $N=25$ steps of nSP$2$S method are necessary and all proposals accepted, despite the $N=357$ steps used by the BlCaSa method with $128$ proposals rejected. Finally, in Figure \ref{multigauss_conf_ess} we compare  the integrators in terms of ESS per unit computational work i.e. ESS$\times h$  for the different values of the step size $h$. As can be seen,  BlCaSa results  very efficient also with value of $h$ greater then $8\times 10^{-3}$; however nSP$2$S turns out to be the best method because its maximum value is the highest among the maximum values reached  by the other two methods.

\begin{figure}[h]\label{multigauss_conf_ess}
 \centering
\includegraphics[width=0.45\linewidth]{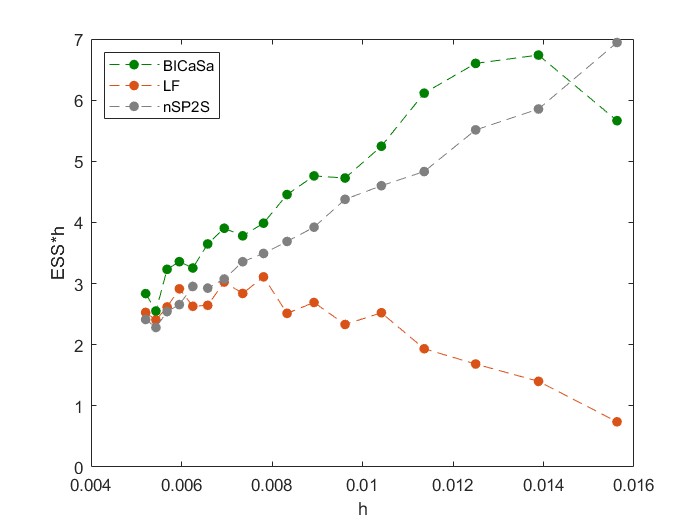}
 \caption{Comparison of the integrators computational cost ESS per time-step/per unit for sampling from Multivariate Gaussian target with $d=256$ by varyng step size $h$. LF (orange points), BlCaSa (green points) and nSP2S (grey points).}
\end{figure}

\subsection{Perturbed Gaussian models}\label{sec:6.3}
In this class we consider models where,  after a suitable change of variables if needed, the potential energy function can be expressed as 
\begin{equation}\label{eq:Uperturbed}
 U(\mathbf{Q})\, =\,  \dfrac{1}{2}\mathbf{Q}^T\,  \, \mathbf{Q} +\epsilon\,  f(\mathbf{Q}), \quad 0 < \epsilon \leq 1.
    \end{equation}
We associate a kinetic energy function $K(\mathbf{P})\,=\, \dfrac{1}{2} \,  \mathbf{P}^T \,  \mathbf{P}$ so that 
the Hamiltonian can be written as 
\begin{equation}\label{eq:Hperturbed}
  K(\mathbf{P}) \,+\,U(\mathbf{Q})\,=\, \dfrac{1}{2}  \mathbf{P}^T \,  \mathbf{P}\,+\, \dfrac{1}{2} \, \mathbf{Q}^T   \mathbf{Q} \,+\,  \epsilon\,  f(\mathbf{Q} )
\end{equation}
and the Hamiltonian system is given by 
\begin{equation}\label{eq:Systemperturbed}
\dfrac{d\mathbf{Q}}{dt}\,=\, \mathbf{P},\qquad 
 \dfrac{d\mathbf{P}}{dt}\,=\, -  \mathbf{Q} \displaystyle - \epsilon \, \nabla_{\mathbf{Q}}\,f(\mathbf{Q}) \end{equation}

\noindent We  apply  the method (\ref{composition_2old}) within the \cref{alg:buildtreenovel} which  provides an energy preserving method for the above Hamiltonian system when $\epsilon=0$. We  investigate the performance of the algorithm with respect to the acceptance rates and to the error in energy at $\epsilon=1$ for  two different specific models: the logistic regression and the Log-Gaussian Cox model.

\subsection{Log-Gaussian Cox model}
As first example we considered, as target, the Log-Gaussian Cox distribution \cite{moller1998log}. For this model the data set is organized in the vector  
$\mathbf{X}\,=\left[X_{1,1},\dots,X_{1,d},X_{2,1},\dots, X_{2,d},\dots X_{d,1}, \dots X_{d,d}\right]^T$,
representing the number of points $X_{i,j}$ in each cell $(i,j)$ of a $d \times d $ dimensional grid in $[0,1]\times [0,1]$. 
The purpose is to sample the variable  $\mathbf{Y}=\left[Y_{1,1},\dots,Y_{1,d},Y_{2,1},\dots, Y_{2,d},\dots Y_{d,1}, \dots Y_{d,d}\right]^T $ from the probability distribution 
given by  $$\mathcal{P}(\mathbf{Y}) = \prod_{i,j=1}^d \exp(X_{i,j} Y_{i,j} - m \exp(Y_{i,j})\,)\, \exp \left(-\dfrac{1}{2}(\mathbf{Y}- \mu \mathbf{1})^T\, S^{-1} (\mathbf{Y}- \mu \mathbf{1}) \right)$$ 
where $m=1/d^2$ represents the area of each cell and the matrix $S$ is given 
\small{
$$
S= \left(  \begin{array}{cccccc}
  T_1 & T_2 & T_3 & \dots  & \dots & T_d\\
     T_2 & T_1 & T_2 & \ldots &  & \vdots\\
      T_3 & T_2 & \ldots & \ldots &  \ldots& \ldots\\
       \vdots & \ldots & \ldots & \ldots & T_2 & T_3\\
       \vdots  &  &  \ldots&  T_2& T_1 & T_2 \\
         T_d & \dots & \dots & T_3 & T_2 & T_1
\end{array}\right),\,\,\, T_i(k,j)= \sigma^2\, e^{-\frac{\sqrt{(1-i)^2+(k-j)^2}}{\beta \, d }},  \,\,\, i,\,k,\,j=1,\dots d. 
$$}
\normalsize
\noindent with $\sigma^2$, $\beta$,   are fixed  parameters and $\mu=\log\left(\displaystyle \sum_{i,j=1}^d X_{i,j}\right) -\sigma^2/2$. The  potential energy function is  defined as $$
\small
\begin{array}{ccc}
   U(\mathbf{Y})\, =\, -\log\left[\mathcal{P}(\mathbf{Y})\right]\, & = & \displaystyle \sum_{i,j=1}^d  m \exp(Y_{i,j}) - X_{i,j} Y_{i,j}\, +\, \dfrac{1}{2}(\mathbf{Y}- \mu \mathbf{1})^T\, S^{-1} (\mathbf{Y}- \mu \mathbf{1}) \\
     & = & \dfrac{1}{2}(\mathbf{Y}- \mu \mathbf{1})^T\, S^{-1} (\mathbf{Y}- \mu \mathbf{1}) \, +\, \displaystyle m \,  \mathbf{1}^T \exp(\mathbf{Y}) -  \mathbf{X}^T \mathbf{Y}.
\end{array}
$$
For the application of the proposed procedure, we take $L$ as the  factor of the 
Cholesky factorization of $S^{-1}$ so that $S^{-1}\,=\,  L^T\, L$. In terms of the novel variable   $\mathbf{Q}:=  L \, (\mathbf{Y}- \mu \mathbf{1})$, the potential energy function can be expressed as in (\ref{eq:Uperturbed}) with ($\epsilon=1$)  
 $$f(\mathbf{Q}):= \displaystyle m \,  \mathbf{1}^T \exp(L^{-1}\mathbf{Q}+ \mu \mathbf{1}) -  \mathbf{X}^T \left(L^{-1}\mathbf{Q}+ \mu \mathbf{1}\right).$$
When we associate a kinetic energy function $K(\mathbf{P})\,=\, \dfrac{1}{2} \, \mathbf{P}^T \,  \mathbf{P}$,
the Hamiltonian can be written as in (\ref{eq:Hperturbed}) and 
 the Hamiltonian system is given by (\ref{eq:Systemperturbed}) 
with $$\nabla_{\mathbf{Q}}\,f(\mathbf{Q}) \,=\, m \,   {L^{-T}}\exp({L}^{-1}\mathbf{Q}+ \mu \mathbf{1}) -  {L^{-T}}\, \mathbf{X}.$$ The sampling values of the original variable are obtained by  exploiting the relation  $\mathbf{Y}={L}^{-1}\mathbf{Q}+ \mu \mathbf{1}$.

\noindent The Log-Gaussian Cox model is particularly relevant as point process to model  presence-only species distribution  \cite{renner2015},  as the case of Scots pines in the Eastern Finland \cite{moller1998log,Christensen2003} or the  spread of the invasive species Eucalyptus sparsifolia, in Australia \cite{renner2015}. In our study, we approach a similar problem of alien plants as the highly competitive woody invasive species \textit{Ailanthus altissima (Mill.) Swingle}, thriving in Murgia Alta Natura 2000 protected area and National Park (southern Italy). \textit{Ailanthus altissima}, also known as tree-of-heaven, is an invasive deciduous plant of Asian origin recognized as one of the most widespread and harmful invasive plants in both  USA \cite{nagendra2013} and Europe \url{(www.europe-aliens.org)} which are causing impoverishment in natural habitats as one of the most important causes of local and regional biodiversity loss, ecosystem degradation, diminishing both abundance and survival of native species \cite{casella2016}.  
\noindent In our tests, we considered as data set the mapping  at very high spatial resolution ($2$ m) of the \textit{Ailanthus altissima} presence obtained by considering multi-temporal remote sensing satellite data and machine learning techniques, based on a two-stage hybrid classification process  \cite{tarantino2019}.
\begin{figure}[h]\label{Ailanto}
  \centering\includegraphics[width=0.6\linewidth]{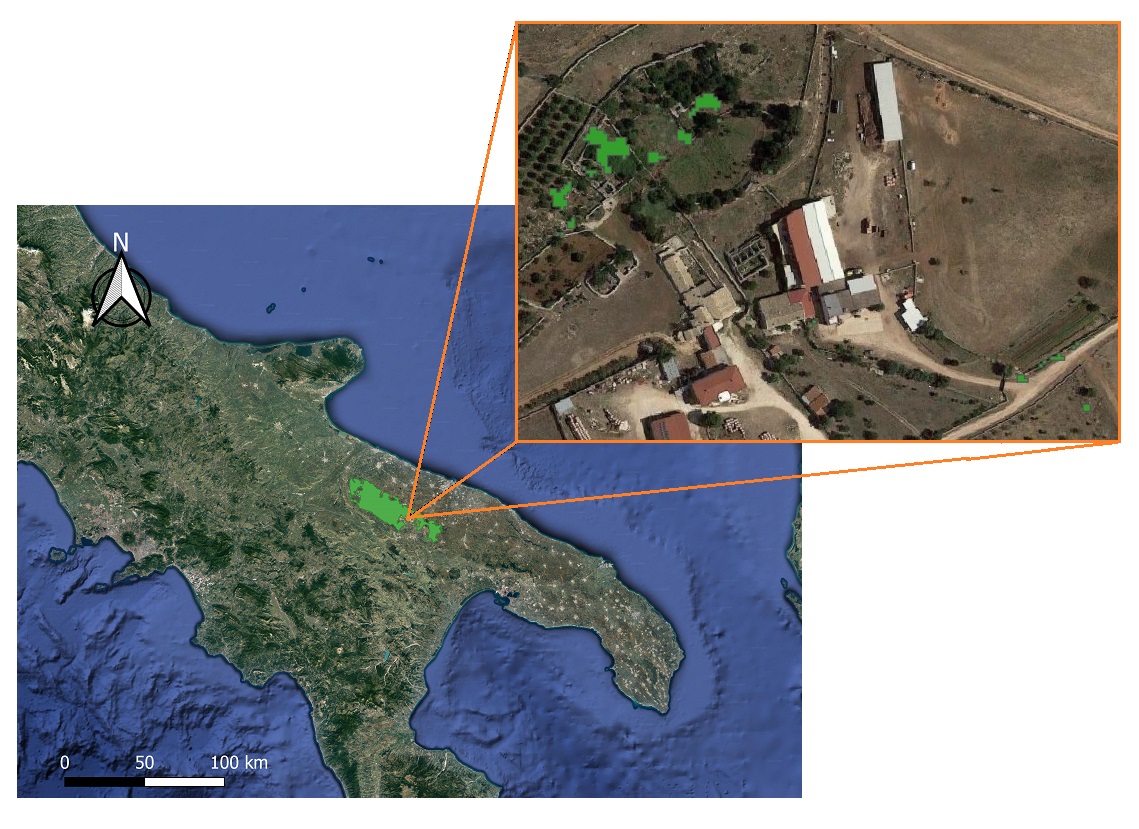}

 \caption{Detail with zoom of the area taken into consideration in Alta Murgia Natura 2000 protected area and National Park (southern Italy)}
\end{figure}
The images considered for the dataset were provided by the European Space
Agency (ESA) under the Data Warehouse 2011-2014 policy within the FP7-SPACE BIO\_SOS project \url{www.biosos.eu} and European LIFE project (LIFE12 218 BIO/IT/000213). In particular, from the entire dataset we extracted a small area containing $185$ trees as shown in Figure \ref{Ailanto}.  After scaling the data between $0$ and $1$ we used a grid size $d=64$ and we calculated the parameters of the Log Gaussian Cox model by using the  methodology of \emph{Moment-based estimation} described in \cite{moraga2013}. The resulting  parameters are $\beta=0.127$, $\sigma^2=3.5881$ and $\mu=\log (185)-\sigma^2/2$. We set the path length $T=3$ and the step size $h$ and $h_b$ in the interval $[0.3,\,0.05]$ with corresponding $N$ ranging from $60$ to $10$. We collect L=5000 Markov chain after 1000 burn-in samples.  In Figure \ref{ExpLogCox}, we have reported the results in terms of acceptance percentage and mean of the energy errors. The horizontal axis of each plot indicates the step size $h$, and $h_b$, for the competitors methods and for our method,  respectively. We clearly see how our method outperforms the competitors, reaching acceptance rates greater than $90\%$ with all $h_b$ values and with an average energy error which always remains very low.
Finally in the Figure \ref{changevar_graymap}, we show the results obtained through the application of the \cref{alg:buildtreenovel}. Starting from $b_{init} =b_{BCS}=0.2113$ (case $2$ of Section \ref{b_selection}) and a reduction factor of   $red\sim 10^{-4}$, after a few steps the method allows to achieve the best configuration for $b$ and $h_b$ yielding the maximum acceptance rate. In the right part of the figure we show the estimate of the intensity map of the Ailanthus trees obtained with the values of $b$ and $h_b$ previously estimated.
\begin{figure}[h]\label{ExpLogCox}
\centering
\subfigure[LF]{\includegraphics[width=0.45\linewidth]{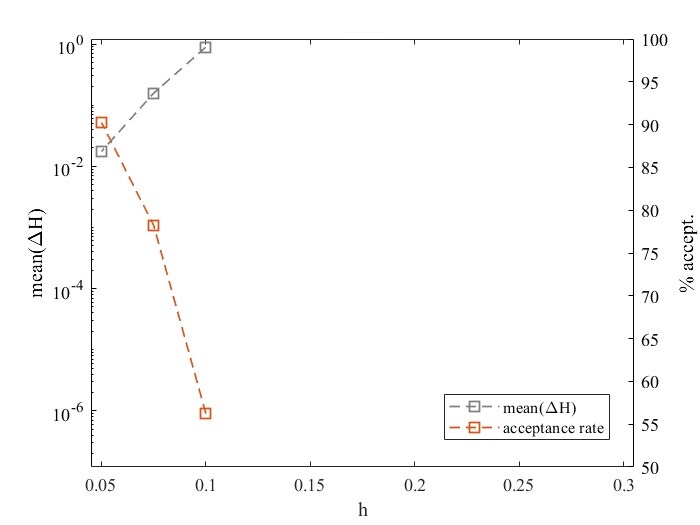}}
\subfigure[BlCaSa]{\includegraphics[width=0.45\linewidth]{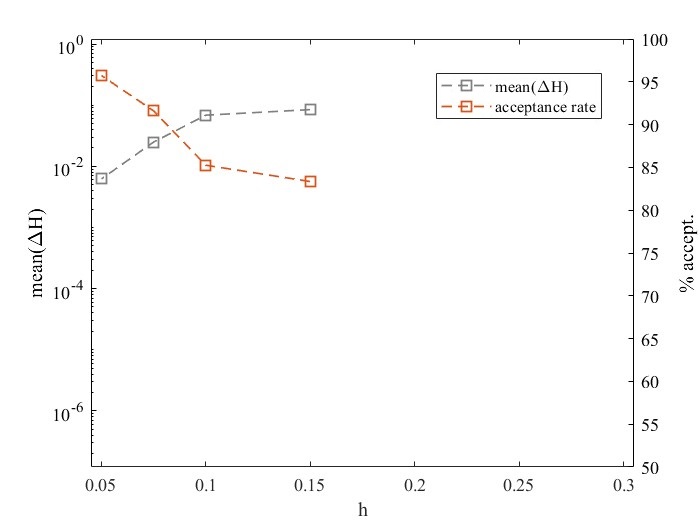}}\\
\subfigure[nSP2S]{\includegraphics[width=0.45\linewidth]{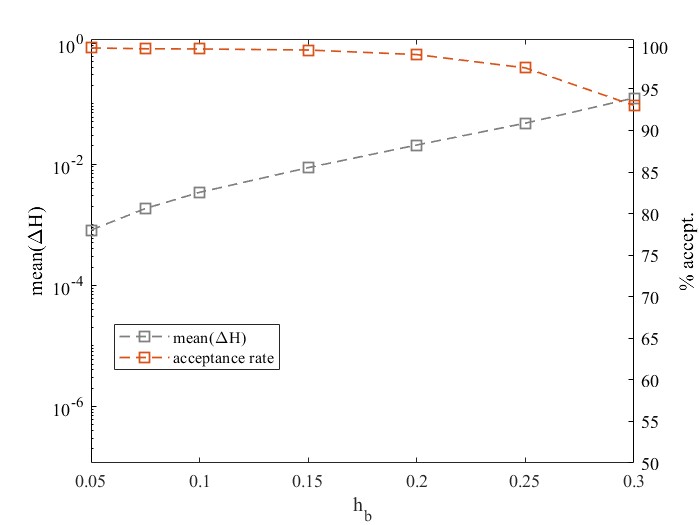}}
\caption{Log-Gaussian Cox target. Acceptance rate (orange sqares), mean of $\Delta H$ (grey sqares) against different step size $h$ for LF (a), BlCaSa (b) and $h_b$ for nSP2S (c).  
}
\end{figure}

\begin{figure}[H]\label{changevar_graymap}
\centering
 \centering\includegraphics[width=0.45\linewidth]{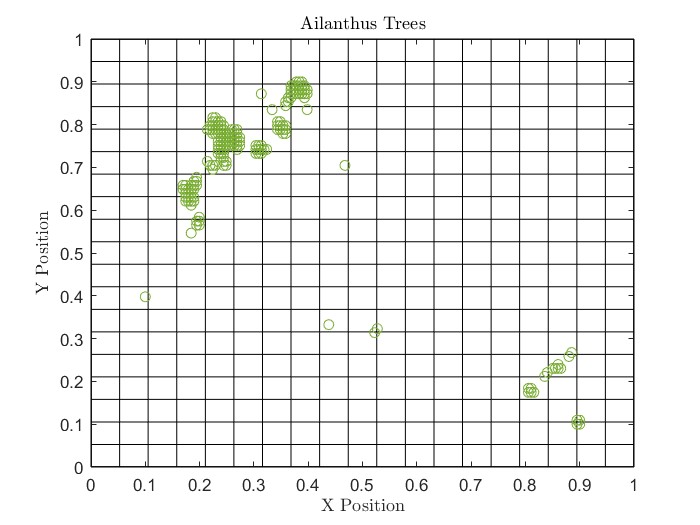}
\subfigure{\includegraphics[width=0.45\linewidth]{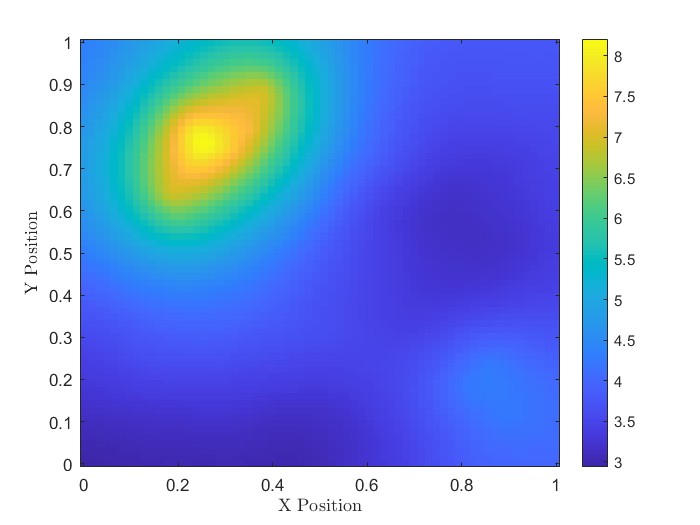}}
\subfigure{\includegraphics[width=0.45\linewidth]{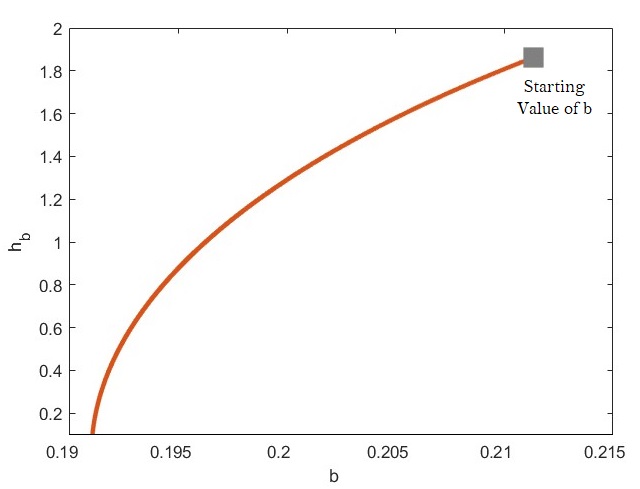}}
\caption{Presence of Ailanthus trees (up, on the left), position of the trees (green points) suitably scaled between $0$ and $1$. Estimated intensity map of Ailhantus obtained with the final choice of $b$ and $h_b$ accordingly to the \cref{alg:buildtreenovel} (up, on the right).  Adaptive choice of $b$ against $h_b$ (down)
}
\end{figure}

\subsection{Logistic regression model}
As second example of  perturbed Gaussian model  distributions we have considered a Bayesian Logistic regression model. By adopting the same notations  in  \cite{thomas2021learning}, we indicate with  $\mathbf{Y}=[Y_1,\ldots Y_n]^T $ 
the $n$-dimensional vector of  
the labels associated to the instances matrix 
$X \in \mathbb{R}^{n\times(d+1)}.$ 
 Let
$\mathbf{x}_k=[X_{k,0},\ldots,X_{k,d}]^T$ the $(d+1)$-dimensional  (column) vector corresponding to the $k$th row of the matrix $X$, for $k=1,\ldots,n.$ The regression coefficients for the $d$  covariates and the intercept are collected in the vector  $\bm \beta=[\beta_0,\beta_1,\ldots,\beta_d]^T$. We specify a multivariate normal prior for $\bm \beta$ with covariance matrix $D=\sigma^2 I$, where $I$ is the $(d+1)$-dimensional identity matrix and $\sigma^2$  is the variance, freely chosen. The purpose is to sample the parameters $\bm \beta$ that follow the distribution:
$$
\small \begin{array}{rcl}
\mathcal{P}(\bm \beta)
&\propto& \exp\Big(\bm \beta^{T}X^{T}\left(\mathbf{Y - \mathbf{1}_n}\right)-\displaystyle \sum_{j=1}^{n}\Big[\log (1+\exp(-\mathbf x_j^T \bm \beta)\Big]\Big)\,\exp\Big(-\frac{1}{2}\bm \beta^T D^{-1}\bm \beta\Big)\\\\
&=& \exp\Big(\bm \beta^{T}X^{T}\mathbf{Y}-\displaystyle \sum_{j=1}^{n}\Big[\log (1+\exp(\mathbf x_j^T \bm \beta)\Big]\Big)\,\exp\Big(-\frac{1}{2}\bm \beta^T D^{-1}\bm \beta\Big)
\end{array}
$$
The potential energy function, in term of the variable $\mathbf{Q}= \bm \beta/ \sigma$ is  defined as:
$$
U(\mathbf{Q})=-\log[\mathcal{P}(\mathbf{Q})]\,=\,\frac{1}{2}\,\mathbf{Q}^T \mathbf{Q}+\sum_{k=1}^{n}\Big[\log (1+ \exp(\sigma \, \mathbf x_k^T \mathbf{Q}))\Big]-\sigma\, \mathbf{Q}^{T}{X}^{T}\mathbf{Y}
$$
Without loosing of generality, we can set $\sigma=1$ as the same results are obtained on the scaled dataset $\tilde X= X/\sigma$. Consequently, it can be expressed as in (\ref{eq:Uperturbed}) ($\epsilon=1$)  with
 $$f(\mathbf{Q}):= \displaystyle \sum_{k=1}^{n}\Big[\log (1+\exp(  \mathbf x_k^T \mathbf{Q}))\Big]- \,  \mathbf{Q}^{T}{X}^{T}\mathbf{Y}.$$
We associate a kinetic energy function $K(\mathbf{P})\,=\, \dfrac{1}{2} \,  \mathbf{P}^T \, \mathbf{P}$ so that 
the Hamiltonian can be written as in (\ref{eq:Hperturbed}) and 
 the Hamiltonian system is given by (\ref{eq:Systemperturbed})
 with 
{\small
$$
\nabla_{\mathbf{Q}}\, f(\mathbf{Q})= -  \, X^{T}\bigg(\mathbf{Y}-\bigg[\frac{\exp(  \,\mathbf x_1^T \mathbf Q)}{1+\exp(  \mathbf x_1^T \mathbf Q)}, \frac{\exp(  \mathbf x_2^T \mathbf Q)}{1+\exp( \mathbf x_2^T \mathbf Q)},\dots, \frac{\exp(\mathbf x_n^T \mathbf Q)}{1+\exp( \mathbf x_n^T \mathbf Q)}\bigg]^T \bigg).
$$} As $\sigma=1$, the sampling values of the original variable are given by $\bm \beta = \sigma \, \mathbf{Q} = \mathbf{Q} $. 
\noindent For this experiment we used the benchmark classification dataset from the UCI repository \cite{bache2013uci}, that consists in different matrices of instances and labels. Here we show  the results  obtained with the Pima Indian dataset, although we also tested the other datasets, Ripley, Heart,
German credit, and Australian credit, obtaining results in line with those shown here.

As is commonly used, we normalize the dataset with $0$ mean and standard deviation $1$ after a scaling procedure of data according to the choosing value of $\sigma$. 
    \begin{figure}[h]\label{logistic_expe}
\centering
  \subfigure[]{\includegraphics[width=0.45\linewidth]{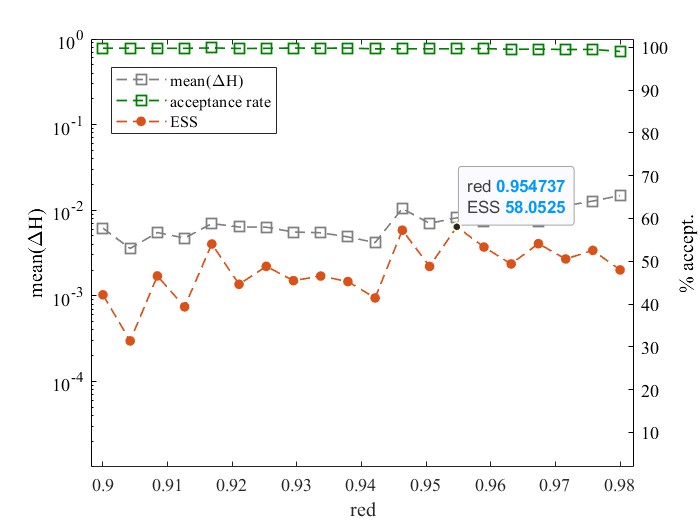}}
  \subfigure[]{\includegraphics[width=0.45\linewidth]{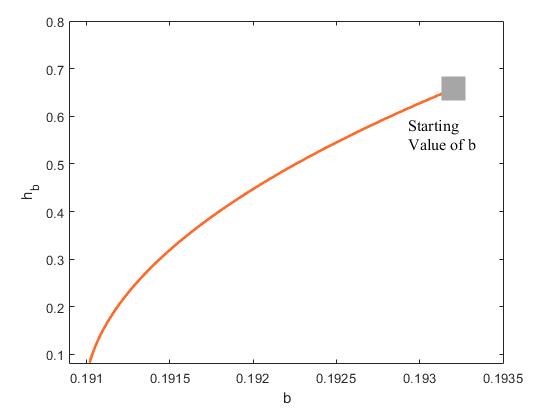}}
 \caption{Logistic Regression target; acceptance rate percentage (green squares), mean of $\Delta H$ (grey squares) and ESS percentage (orange points) measuring with different values of $red$ in the Algorithm \ref{alg:buildtreenovel} (a). Adaptive choice of b against $h_b$ (b).}
\end{figure}
We test the performance of the proposed integrator built on the proposed adaptive approach \ref{alg:buildtreenovel} by setting, the path length $T = 3$, $b_{max}=0.1932$ (case $4$ of Section \ref{b_selection}). In all the experiment we collect $L=5000$ Markov chain after 1000 burn-in samples and for each iteration we randomize the time-step $N=T/h_{b}$ by allowing $\pm 10\% $ to avoid a low number of ESS \cite{neal2011mcmc}. 
\begin{figure}[H]\label{logistic_expe_hist}
\centering
 \includegraphics[width=0.8\linewidth]{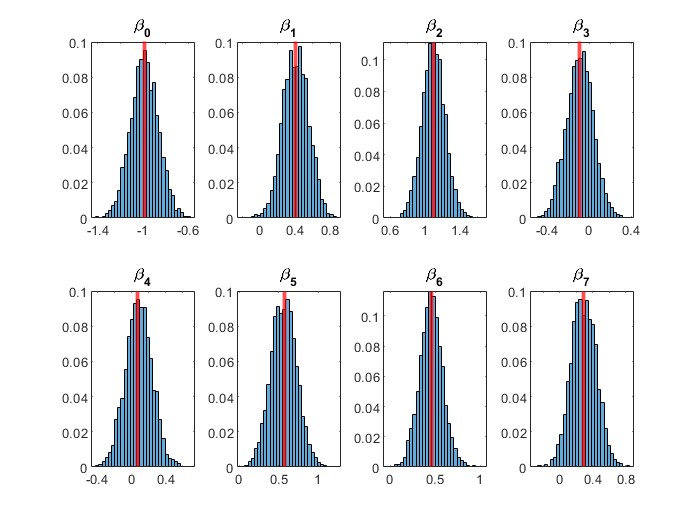}
 \caption{Logistic Regression with Pima Indian dataset. Analisys of posterior estimates HMC-nSP2S (histograms) and frequentist estimates (vertical lines) with glm}
\end{figure}
For this experiment we perform the tests by varying the reduction parameter $red$ within the range $(0.9,0.98)$ and measuring for each of these values the efficiency in term of acceptance percentage, mean of energy errors end ESS calculated as the mean of the ESS on each regression coefficients $\mathbf{q}^{(i)}=[\beta_0^{(i)},\beta_1^{(i)},\ldots\beta_d^{(i)}]$ with $i=1000,\ldots,5000$. The results obtained are shown in Figure \ref{logistic_expe} (a). For each values of $red$ the adaptive algorithm reaches an high percentage of AR. In particular $red=0.954737$ represents the best values that ensures a high acceptance rate along with a high percentage of ESS. Given this reduction value, the Figure \ref{logistic_expe} (b) shows how $b$ changes with respect to $h_b$ reaching, automatically, the best configuration for the sampler. Moreover the estimated samples $\bm \beta$ are deemed to be in accordance with frequency estimates calculated with the generalized linear model (glm) \cite{Nelder1972}. This is clearly shown in Figure \ref{logistic_expe_hist} where the frequentist estimates (red line) falls in the central location of the histograms of $\beta_0,\beta_1,\ldots,\beta_7$.
\section{Conclusions}\label{sec:conclusions}The very recent research literature on searching for efficient volume-preserving and reversible   integrators able to  replace the St\"ormer Verlet  in the practical implementation of the HMC method, uses as  yardstick the ability of the numerical algorithms  to reduce  the expectation of the energy error variable when applied to univariate and multivariariate Gaussian distributions \cite{calvo2019hmc},\cite{blanes2014numerical}. However, none of the above studies makes an explicit reference to the possibility of properly selecting the parameters in order to exactly  preserve the Hamiltonian in order to have zero as expectation value of the energy error.

\noindent In this paper, we analyzed a representation of  linear maps corresponding to trajectories evaluated by mean of symplectic reversible splitting schemes for sampling from Gaussian distributions. In our formulation  (different from the one proposed in \cite{blanes2014numerical}), the expression of the expectation value of the energy error gives evidence of the role of the quantity which is  solely responsible for the distortion in calculated energy. Minimizing this quantity results in reducing the number of rejections in practical implementation of the HMC algorithm; consequently, methods with this quantity equal to zero  are optimal for Gaussian univariate and multivariate distributions in terms of  accepted proposals.

\noindent Within the one-parameter family of second order  splitting methods  considered in this paper, we show that a properly selection of the step size  allows to retain the energy for univariate and multivariate Gaussian distributions.  As all proposals are accepted by construction, the proposed method outperforms the numerical competitors given in \cite{calvo2019hmc,blanes2021symmetrically}  which are neither optimal for Gaussian distributions (as they are not energy-preserving methods) nor less expensive than the ones considered here.

\noindent For more general distributions, which can be interpreted as perturbed  Gaussian distributions,  the same  criterion for the selection of the step size is proposed. The resulting nice  performances are  also enhanced by the application of an adaptive selection of the parameter $b$ which detects one method in the family of the one-parameter splitting integrators.  
Specifically, starting from a suitable large initial value for $b$, we reduce its value  of a given percentage each time a sample is not accepted. 

\noindent In order to validate the effectiveness of the new approach we tested the algorithm also for general classes of target distribution 
as the the Log-Gaussian Cox model and the Bayesian logistic regression, in particular the first one 
is relevant as point process to model  presence-only species distribution  of invasive species \cite{renner2015}. In our test, we considered as data set the mapping  at very high spatial resolution of the \textit{Ailanthus altissima} presence in Alta Murgia National Park \cite{tarantino2019}. 

\noindent Our promising preliminar results  require to exploit the use of more robust criteria for an  adaptive method's selection in one-parameter families of integrators having the objective of reducing the number of rejections. This can be the subject of a future research direction.   

\section*{Acknowledgments} Cristiano Tamborrino  has been supported by LifeWatch Italy through
the project LifeWatchPLUS (CIR-01-00028).
\printbibliography 

@article{pace2015splitting,
  title={Splitting schemes and energy preservation for separable Hamiltonian systems},
  author={Pace, Brigida and Diele, Fasma and Marangi, Carmela},
  journal={Mathematics and Computers in Simulation},
  volume={110},
  pages={40--52},
  year={2015},
  publisher={Elsevier}
}

@article{blanes2014numerical,
  title={Numerical integrators for the {H}ybrid {M}onte {C}arlo method},
  author={Blanes, Sergio and Casas, Fernando and Sanz-Serna, JM},
  journal={SIAM Journal on Scientific Computing},
  volume={36},
  number={4},
  pages={A1556--A1580},
  year={2014},
  publisher={SIAM}
}

@article{hairer2003geometric,
  title={Geometric numerical integration illustrated by the Stormer-Verlet method},
  author={Hairer, Ernst and Lubich, Christian and Wanner, Gerhard and others},
  journal={Acta numerica},
  volume={12},
  number={12},
  pages={399--450},
  year={2003}
}

@article{predescu2012computationally,
  title={Computationally efficient molecular dynamics integrators with improved sampling accuracy},
  author={Predescu, Cristian and Lippert, Ross A and Eastwood, Michael P and Ierardi, Douglas and Xu, Huafeng and Jensen, Morten {\O} and Bowers, Kevin J and Gullingsrud, Justin and Rendleman, Charles A and Dror, Ron O and others},
  journal={Molecular Physics},
  volume={110},
  number={9-10},
  pages={967--983},
  year={2012},
  publisher={Taylor \& Francis}
}

@article{joo2000instability,
  title={Instability in the molecular dynamics step of a hybrid Monte Carlo algorithm in dynamical fermion lattice QCD simulations},
  author={Joo, Balint and Pendleton, Brian and Kennedy, Anthony D and Irving, Alan C and Sexton, James C and Pickles, Stephen M and Booth, Stephen P and UKQCD Collaboration and others},
  journal={Physical Review D},
  volume={62},
  number={11},
  pages={114501},
  year={2000},
  publisher={APS}
}

@article{metropolis1953equation,
  title={Equation of state calculations by fast computing machines},
  author={Metropolis, Nicholas and Rosenbluth, Arianna W and Rosenbluth, Marshall N and Teller, Augusta H and Teller, Edward},
  journal={The journal of chemical physics},
  volume={21},
  number={6},
  pages={1087--1092},
  year={1953},
  publisher={American Institute of Physics}
}

@article{alder1959studies,
  title={Studies in molecular dynamics. I. General method},
  author={Alder, Berni J and Wainwright, Thomas Everett},
  journal={The Journal of Chemical Physics},
  volume={31},
  number={2},
  pages={459--466},
  year={1959},
  publisher={American Institute of Physics}
}

@article{duane1987hybrid,
  title={Hybrid monte carlo},
  author={Duane, Simon and Kennedy, Anthony D and Pendleton, Brian J and Roweth, Duncan},
  journal={Physics letters B},
  volume={195},
  number={2},
  pages={216--222},
  year={1987},
  publisher={Elsevier}
}

@article{neal2011mcmc,
  title={MCMC using Hamiltonian dynamics},
  author={Neal, Radford M and others},
  journal={Handbook of markov chain monte carlo},
  volume={2},
  number={11},
  pages={2},
  year={2011}
}

@article{takaishi2006testing,
  title={Testing and tuning symplectic integrators for the hybrid Monte Carlo algorithm in lattice QCD},
  author={Takaishi, Tetsuya and De Forcrand, Philippe},
  journal={Physical Review E},
  volume={73},
  number={3},
  pages={036706},
  year={2006},
  publisher={APS}
}

@book{leimkuhler2004simulating,
  title={Simulating hamiltonian dynamics},
  author={Leimkuhler, Benedict and Reich, Sebastian},
  number={14},
  year={2004},
  publisher={Cambridge university press}
}

@article{leimkuhler2013robust,
  title={Robust and efficient configurational molecular sampling via Langevin dynamics},
  author={Leimkuhler, Benedict and Matthews, Charles},
  journal={The Journal of chemical physics},
  volume={138},
  number={17},
  pages={05B601\_1},
  year={2013},
  publisher={American Institute of Physics}
}

@article{mclachlan1995numerical,
  title={On the numerical integration of ordinary differential equations by symmetric composition methods},
  author={McLachlan, Robert I},
  journal={SIAM Journal on Scientific Computing},
  volume={16},
  number={1},
  pages={151--168},
  year={1995},
  publisher={SIAM}
}

@article{calvo2019hmc,
  title={HMC: avoiding rejections by not using leapfrog and some results on the acceptance rate},
  author={Calvo, Mari Paz and Sanz-Alonso, Daniel and Sanz-Serna, JM},
  journal={arXiv preprint arXiv:1912.03253},
  year={2019}
}

@article{blanes2021symmetrically,
  title={Symmetrically processed splitting integrators for enhanced Hamiltonian Monte Carlo sampling},
  author={Blanes, Sergio and Calvo, Mari Paz and Casas, Fernando and Sanz-Serna, Jes{\'u}s Mar{\'\i}a},
  journal={SIAM Journal on Scientific Computing},
  volume={43},
  number={5},
  pages={A3357--A3371},
  year={2021},
  publisher={SIAM}
}

@article{moller1998log,
  title={Log gaussian cox processes},
  author={M{\o}ller, Jesper and Syversveen, Anne Randi and Waagepetersen, Rasmus Plenge},
  journal={Scandinavian journal of statistics},
  volume={25},
  number={3},
  pages={451--482},
  year={1998},
  publisher={Wiley Online Library}
}

@incollection{marangi2020mathematical,
  title={Mathematical tools for controlling invasive species in Protected Areas},
  author={Marangi, Carmela and Casella, Francesca and Diele, Fasma and Lacitignola, Deborah and Martiradonna, Angela and Provenzale, Antonello and Ragni, Stefania},
  booktitle={Mathematical Approach to Climate Change and its Impacts},
  pages={211--237},
  year={2020},
  publisher={Springer}
}

@article{baker2019optimal,
  title={Optimal control of invasive species through a dynamical systems approach},
  author={Baker, Christopher M and Diele, Fasma and Lacitignola, Deborah and Marangi, Carmela and Martiradonna, Angela},
  journal={Nonlinear Analysis: Real World Applications},
  volume={49},
  pages={45--70},
  year={2019},
  publisher={Elsevier}
}

@article{baker2018optimal,
  title={Optimal spatiotemporal effort allocation for invasive species removal incorporating a removal handling time and budget},
  author={Baker, Christopher M and Diele, Fasma and Marangi, Carmela and Martiradonna, Angela and Ragni, Stefania},
  journal={Natural Resource Modeling},
  volume={31},
  number={4},
  pages={e12190},
  year={2018},
  publisher={Wiley Online Library}
}

@article{moraga2013,
author = {Diggle, Peter and Moraga, Paula and Rowlingson, Barry and Taylor, Benjamin},
year = {2013},
month = {12},
pages = {},
title = {Spatial and Spatio-Temporal Log-Gaussian Cox Processes: Extending the Geostatistical Paradigm},
volume = {28},
journal = {Statistical Science},
doi = {10.1214/13-STS441}
}

@article{renner2015,
author = {Renner, Ian W. and Elith, Jane and Baddeley, Adrian and Fithian, William and Hastie, Trevor and Phillips, Steven J. and Popovic, Gordana and Warton, David I.},
year = {2015},
month = {},
pages = {366-379},
title = {Point process models for presence-only analysis},
volume = {6},
journal = {Methods in Ecology and Evolution},
doi = {10.1111/2041-210X.12352}
}

@article{tarantino2019,
author = {Tarantino, Cristina and Casella, Francesca and Adamo, Maria and Lucas, Richard and Beierkuhnleind, Carl and Blonda, Palma},
year = {2019},
month = {},
pages = {90-103},
title = {Ailanthus altissima mapping from multi-temporal very high resolution satellite images},
volume = {147},
journal = {ISPRS Journal of Photogrammetry and Remote Sensing},
doi = {10.1016/j.isprsjprs.2018.11.013}
}

@article{nagendra2013,
author = {Nagendra, Harini and Lucas, Richard and Honrado, Joao P. and Jongman, Robert and Tarantino, Cristina and Adamo, Maria and Mairota, Palma},
year = {2013},
month = {},
pages = {45-59},
title = {Remote sensing for conservation monitoring: assessing protected
areas, habitat extent, habitat condition, species diversity and threats.},
volume = {33},
journal = {Ecological Indicators},
doi = {10.1016/j.ecolind.2012.09.014}
}

@inproceedings{casella2016,
  title={Restoration of areas infested by A. altissima in the Alta Murgia National Park: experience within a LIFE project.},
  author={Casella, Francesca and Vurro, Michele and Boari, A.G.},
  booktitle={Seventh International Weed Science Congress},
  pages={215},
  year={2016},
  organization={Prague, Czech Republic}
}

@article{thomas2021learning,
  title={Learning Hamiltonian Monte Carlo in R},
  author={Thomas, Samuel and Tu, Wanzhu},
  journal={The American Statistician},
  volume={75},
  number={4},
  pages={403--413},
  year={2021},
  publisher={Taylor \& Francis}
}

@article{bache2013uci,
  title={UCI machine learning repository},
  author={Bache, Kevin and Lichman, Moshe},
  year={2013},
  publisher={Irvine, CA, USA}
}

@article{lacitignola2015dynamical,
  title={Dynamical scenarios from a two-patch predator--prey system with human control--Implications for the conservation of the wolf in the Alta Murgia National Park},
  author={Lacitignola, Deborah and Diele, Fasma and Marangi, Carmela},
  journal={Ecological modelling},
  volume={316},
  pages={28--40},
  year={2015},
  publisher={Elsevier}
}

@article{MARTINO2017386,
title = {Effective sample size for importance sampling based on discrepancy measures},
journal = {Signal Processing},
volume = {131},
pages = {386-401},
year = {2017},
issn = {0165-1684},
doi = {https://doi.org/10.1016/j.sigpro.2016.08.025},
url = {https://www.sciencedirect.com/science/article/pii/S0165168416302110},
author = {Luca Martino and Víctor Elvira and Francisco Louzada},
}

@article{Christensen2003,
author = {Christensen, Ole and Roberts, Gareth and Rosenthal, Jeffrey},
year = {2003},
month = {02},
pages = {},
title = {Scaling Limits for the Transient Phase of Local Metropolis-Hastings Algorithms},
volume = {67},
journal = {Journal of the Royal Statistical Society Series B},
doi = {10.1111/j.1467-9868.2005.00500.x}
}

@article{Nelder1972,
 ISSN = {00359238},
 URL = {http://www.jstor.org/stable/2344614},
 author = {J. A. Nelder and R. W. M. Wedderburn},
 journal = {Journal of the Royal Statistical Society. Series A (General)},
 number = {3},
 pages = {370--384},
 publisher = {[Royal Statistical Society, Wiley]},
 title = {Generalized Linear Models},
 volume = {135},
 year = {1972}
}

\end{document}